\documentclass[11pt]{amsart}
\usepackage{pstricks} 
\usepackage{color}
\usepackage{todonotes}
\usepackage{tikz}
\usepackage{circuitikz}
\usetikzlibrary{positioning,patterns}
\usetikzlibrary{decorations.pathreplacing,decorations.markings}
\usepackage[pagewise]{lineno}
\usepackage{accents}
\usepackage{amsthm, amsmath}
\usepackage[all]{xy}
\usepackage{amssymb,latexsym}
\usepackage{graphicx}
\usepackage{enumitem} 
\usepackage{multirow}

\renewcommand{\phi}{\varphi}
\newcommand{\sns}[3]{\langle#1|#2|\cdots|#3\rangle}

\newcommand{\zed}{\ensuremath{\mathbb{Z}}}

\newcommand{\reals}{\mathbb{R}}

\newcommand{\qsal}{\mathcal S}
\newcommand{\qncp}{\mathcal N}

\newtheorem{theorem}{Theorem}[section]
\newtheorem{proposition}[theorem]{Proposition}
\newtheorem{corollary}[theorem]{Corollary}
\newtheorem{lemma}[theorem]{Lemma}
\theoremstyle{definition}

\newtheorem{definition}[theorem]{Definition}
\newtheorem{remark}[theorem]{Remark}
\newtheorem{example}[theorem]{Example}

\numberwithin{equation}{section}
\begin{document}
\title[Triangulating the Permutahedron]
{Triangulating the Permutahedron}
\author[Brady]{Thomas~Brady}
\email{tomglbrady@gmail.com}
\author[Delucchi]{Emanuele Delucchi}
\address{IDSIA USI-SUPSI, University of applied arts and sciences of Southern Switzerland, Lugano, Switzerland}
\email{emanuele.delucchi@supsi.ch}
\author[Watt]{Colum~Watt}
\address{School of Mathematics and Statistics\\
Technological University Dublin\\
Dublin D07 ADY7\\
Ireland}
\date{\today}

\begin{abstract}
For a finite irreducible real reflection group, $W$, we triangulate its permutahedron and use this to give an explicit homotopy equivalence between two known classifying spaces for the associated Artin group, $B(W)$.  In the process, we characterise the Bruhat intervals 
from $w_1$ to $w_2$ in $W$, where $w_1^{-1}w_2$ is a 
Coxeter element.
\end{abstract}
\maketitle
\section{Introduction}
\label{intro}
For $\pi$ an Artin group of finite type with reflection group $W$,  there are two well-studied 
$K(\pi, 1)$'s.  The Salvetti $K(\pi, 1)$, $\qsal(W)$, is obtained by making identifications on the corresponding $W$-permutahedron (\cite{Salv94}) and 
derives from a general construction (\cite{Salvetti87}) of a model for the homotopy type of the complement of a general complexified real hyperplane arrangement.
On the other hand, the non-crossing partition $K(\pi, 1)$, 
$\qncp(W)$, is obtained by making identifications on the order complex of the corresponding non-crossing partition lattice for $W$ (\cite{Bessis}, \cite{BWKayPie}) and is a special case of a classifying space  construction (\cite{BessisAnnals}) for the braid group of a general finite complex reflection group.
The interplay between analogues of these two $K(\pi,1)$'s has led to recent progress in the $K(\pi,1)$ conjecture for Artin groups of non-finite type \cite{PaoliniSalvetti,DelucchiPaoliniSalvetti}.

\vskip .2cm

Homotopy theory ensures the existence of a homotopy equivalence from the Salvetti $K(\pi, 1)$, $\qsal(W)$, to the non-crossing partition $K(\pi, 1)$, $\qncp(W)$.  
In fact, the proof of Proposition 1B.9 of \cite{H} shows that such an equivalence can be constructed 
recursively on skeleta.  This construction begins with an inclusion of the 
$1$-skeleton of $\qsal(W)$ into the $1$-skeleton of 
$\qncp(W)$, since the standard generators of the Artin group of finite type are a subset of the 
non-crossing partition generators, 
and concludes with an expression for the image of the
single top-dimensional cell  of $\qsal(W)$ as a combination of  top-dimensional cells of $\qncp(W)$.
In the universal covers, this gives a triangulation of the permutahedron.  
We describe such a triangulation  in the case of an irreducible group, $W$, and show that it induces 
a homotopy equivalence from $\qsal(W)$ to $\qncp(W)$.  
Our triangulation uses a new characterisation of  the Bruhat intervals 
from $w_1$ to  $w_2$, where $w_1^{-1}w_2$ is a 
Coxeter element.
\vskip .2cm
This paper has considerable overlap with \cite{DSBW}, which cites an early report on this work (\cite[Remark 1.1]{DSBW}).   In \cite{DSBW}, a height function is used to show that the special $n$-simplices of section \ref{triangfacets} below determine a  triangulation of the permutahedron, without the assumption (which we will make) that the factorisation of the Coxeter element is bipartite.
\vskip .2cm
The paper is organised as follows.  In section \ref{background}, 
we collect some facts about non-crossing partitions, roots, simplices in $\reals^n$ and the two $K(\pi, 1)$'s.    The
notion of a special $n$-simplex (a top-dimensional simplex
in our triangulation) is introduced in section \ref{triangfacets}. Special $n$-simplices and the above-mentioned Bruhat intervals are also characterised in this section.
Section \ref{permfacechar} characterises the intersections 
of special $n$-simplices with the 
faces of the permutahedron 
while section \ref{dihedral} illustrates the construction in the rank $2$ case.  
In section \ref{panelschar}, we  determine 
the incidences of facets of special $n$-simplices
with other   special $n$-simplices
and with the boundary of the permutahedron.  
In section \ref{triangcrit}, we verify (using criteria from \cite{DRS}) that our set of simplices does indeed give a triangulation and we complete the 
construction of the homotopy equivalence.

\subsection*{Acknowledgements} 
This work started during a  
visit of ED to Dublin City University, whose hospitality and support is acknowledged. 

\section{Background}
\label{background}
\subsection{Non-crossing partitions}\label{NCPsection}
The reflection in the hyperplane orthogonal to a non-zero vector $\theta$ in $\reals^n$ is denoted by $R(\theta)$ and is given by 
 \[R(\theta)v = v-2\left[\frac{\theta\cdot v}{\theta\cdot \theta}\right]\theta \quad \mbox{for all }\ v \in \reals^n.\]
 It follows that, for vectors $\theta_1, \dots , \theta_d$, 
 \begin{equation}\label{multiplereflections}
 R(\theta_1) \cdots R(\theta_d)v = v+a_1\theta_1+\dots + a_d\theta_d
 \end{equation}
 for some $a_1, \dots , a_d \in \reals$.\\
Let $W$ be a finite, irreducible, real reflection group acting effectively on $\reals^n$ and   
let $C$ be a fixed fundamental
chamber with inward unit normals given by the simple roots 
$\alpha_1, \dots , \alpha_n$.  
By \cite{Steinberg}, we can assume these roots have been ordered so that 
$\Pi_1 = \{\alpha_1, \dots , \alpha_k\}$ and $\Pi_2 = \{\alpha_{k+1}, \dots , \alpha_n\}$  
are orthonormal sets.  Set $s_i = R(\alpha_i)$ so that $S = \{s_1, \dots , s_n\}$ is the set of simple reflections and let $c = s_1\cdots s_n$ be the corresponding Coxeter element.  Denote by $h$ the order of $c$ in $W$.  
The set, $T$, of all reflections in $W$ is given by
$T = \{wsw^{-1}\mid  s \in S, w \in W\}$.  
Let  $\{\beta_1, \dots , \beta_n\}$ be the dual basis to $\{\alpha_1, \dots , \alpha_n\}$  
and set $v_0 = \beta_1+  \dots + \beta_n$. Since $v_0 \cdot \alpha_i = 1$ for each $i$,  $v_0$ lies in the interior of the fundamental chamber $C$ and $v_0 \cdot \theta > 0$, for each positive root $\theta$.   For $1\le i \le n$, let $W_i = \langle s_1, \dots , s_{i-1}, s_{i+1}, \dots , s_n \rangle$. 
\vskip .2cm
The total or absolute reflection length function, $l_T: W \to \zed$, is given by
\[l_T(w) = \min\{m \mid w = t_1t_2\cdots t_m \ \mbox{with} \ t_j \in T\}.\]  
The total reflection order on $W$ is defined by
\[u \leq_T w \ \ \mbox{ if and only if } \ \ \ l_T(u) + l_T(u^{-1}w) = l_T(w)\]
and the lattice of non-crossing partitions,  $L$, is the
subset of $W$ consisting of those elements $w$ satisfying $w \leq_T c$.  

The standard length function, $l_S: W \to \zed$, is given by
\[l_S(w) = \min\{k \mid w = r_1r_2\cdots r_{k} \ \mbox{with} \ r_{i} \in S\}. \]
An expression $w = r_1r_2\cdots r_{d}$ for $w$ with $r_i \in S$ is called reduced if $d = l_S(w)$.  
The Bruhat order on $W$ can be defined by $u \le w$ if, for every reduced expression $w = r_1r_2\cdots r_{d}$ for $w$, there is a subword of $r_1r_2\cdots r_{d}$ which is an expression for $u$.    Background on root systems, $l_S$ and the Bruhat order can be found in \cite{Hum90}.

Define $\rho_i = s_1\cdots s_{i-1}\alpha_i$, where $s_{n+i} := s_i$ 
and $\alpha_{n+i} := \alpha_i$.  From \cite[Theorem 6.3]{Steinberg},  $\Phi = \{\rho_1, \dots , \rho_{nh}\}$ is the set of all unit roots and  
$\Phi_+ = \{\rho_1, \dots , \rho_{nh/2}\}$ is the set of positive unit roots.   In particular, $|T| = nh/2$. 

Since $\Pi_1$ and $\Pi_2$ are orthogonal sets 
the first $n$ positive roots are  
\[\alpha_1, \dots , \alpha_k, -c\alpha_{k+1}, \dots , -c\alpha_n,\]
in that order.  
It can be shown (using the ideas in the proof of Proposition~3.18 of \cite{Hum90}) that the last $n$ positive roots are 
$-c^{-1}\alpha_{1}, \dots , -c^{-1}\alpha_k$ (possibly permuted)  followed by $ \alpha_{k+1}, \dots , \alpha_n$ (possibly permuted).  
\vskip .2cm
As in \cite{Lattices}, we define the vectors $\mu_1, \dots , \mu_{nh}$ by 
\[\mu_i = s_1\cdots s_{i-1}\beta_i, \ \mbox{where} \ s_{n+i} := s_i \ \mbox{and} \ \beta_{n+i} := \beta_i.\]
We recall that $\mu_i = \mu(\rho_i)$, where $\mu$ is the linear transformation given by $\mu = -2(c-I)^{-1}$ and note that $\mu$ commutes with $c$.  Furthermore, for any unit length root $\theta$, $\mu(\theta)$ is the unique 
vector which lies in the one-dimensional fixed space of $R(\theta)c$ and satisfies $\mu(\theta)\cdot \theta = 1$ (see Corollary~4.2 of \cite{Lattices}).   It follows that  
$R(\theta)\mu(\theta) = c\mu(\theta)$ and that $c\mu(\theta)\cdot\theta = -1$.

\subsection{Bases of roots and their duals}
\label{rootsduals}
Suppose $\theta_1, \dots , \theta_n$ are unit length roots with 
$c = R(\theta_1)R(\theta_2)\cdots R(\theta_n)$.  Then  $\{\theta_1, \dots , \theta_n\}$ is a basis for $\reals^n$ since $c$ has trivial fixed space.  Furthermore
\begin{equation}\label{E-rollright}
c = R(c\theta_{i})\cdots R(c\theta_{n})R(\theta_1)\cdots R(\theta_{i-1})
\end{equation}
and 
\begin{equation}\label{E-rollleft}
c = R(\theta_{i})\cdots R(\theta_n)R(c^{-1}\theta_1)\cdots R(c^{-1}\theta_{i-1}).
\end{equation}
\begin{proposition}\label{rollrightleft}  (a) The vector $\mu(\theta_i)$  is orthogonal to each of the 
vectors $c^{-1}\theta_1, \dots , c^{-1}\theta_{i-1}, \theta_{i+1}, \dots , \theta_{n}$.\\
(b)  The vector $\theta_i$  is orthogonal to each of the 
vectors $\mu(\theta_1), \dots , \mu(\theta_{i-1}),$ $ \mu(c\theta_{i+1}), \dots , \mu(c\theta_{n})$.
\end{proposition}
\begin{proof}  If  $\phi_1, \phi_2, \dots , \phi_n$ are unit length roots and $c = R(\phi_1)R(\phi_2)\cdots R(\phi_n)$ then the vector  $\mu(\phi_1)$ lies in the one-dimensional fixed subspace of $R(\phi_1)c$, which is 
equal to $ \phi_2^{\perp}\cap \dots \cap \phi_n^{\perp}$.
Thus $\mu(\phi_1) \cdot \phi_i = 0$ for $2\le i \le n$.  Part (a) follows from 
equation~(\ref{E-rollleft}).  Part (b) follows from part (a), the 
commuting of $c$ with $\mu$ and the fact that $c$ is an orthogonal transformation.
\end{proof} 
 For each $i = 1,\dots,n$,  define the roots $\epsilon_i$ and $\tau_i$
by
\[\epsilon_{i} = R(\theta_1)\cdots R(\theta_{i-1})\theta_{i}
\mbox{\ \ \ and  \ \ \ } \tau_{i} = R(\theta_n)\cdots R(\theta_{i+1})\theta_{i}.\]
We note that  $\epsilon_i = -c \tau_i$ and that 
\[ c = R(\epsilon_i)[R(\theta_1)\cdots \widehat{R(\theta_i)}\cdots R(\theta_n)]
=  [R(\theta_1)\cdots \widehat{R(\theta_i)} \cdots R(\theta_n)]R(\tau_i).\]
It also follows from the formulae for $\epsilon_i$ and $\tau_i$ and equation~(\ref{multiplereflections}) that both $\{\epsilon_1, \dots , \epsilon_n\}$ 
and $\{\tau_1, \dots , \tau_n\}$ are bases for $\reals^n$.
\begin{example}\label{firstnlastn}
Set $\theta_i = \alpha_i$ for $1\le i\le n$ so that we are considering the factorisation $c = R(\alpha_1)\cdots R(\alpha_n)$.  Then $\epsilon_j = \rho_{j}$ for $1\le j \le n$.   
Since $c\rho_i = \rho_{i+n}$ modulo $nh$, it follows that the first $n$ positive roots, namely $\alpha_1, \dots , \alpha_k, -c\alpha_{k+1}, \dots ,-c\alpha_n$,  are precisely 
the positive unit roots $\theta$ satisfying $\ [c^{-1}\theta]\cdot v_0 < 0$.
Since $\Pi_1$ and $\Pi_2$ are orthogonal sets, we compute
\[\tau_i= \left\{\begin{array}{cl}
    -c^{-1}\alpha_i & \mbox{for}\ 1\le i\le k, \\
    \alpha_i & \mbox{for}\ k+1\le i \le n.
\end{array}\right.\]
These are the last $n$ positive roots and are precisely 
the positive unit roots $\theta$ satisfying $\ [c\theta]\cdot v_0 < 0$.
\end{example}
\begin{proposition}\label{dualbases}
(a)  $\{\mu(\epsilon_1), \dots , \mu(\epsilon_n)\}$ is the dual basis to $\{\theta_1, \dots, \theta_n\}$. \\
(b) $\{\mu(\theta_1), \dots , \mu( \theta_n)\}$ is the dual basis to $\{\tau_1, \dots, \tau_n\}$. \\
(c)  $\{-\mu(c \theta_1), \dots , -\mu(c \theta_n)\}$ is the dual basis to $\{\epsilon_1, \dots, \epsilon_n\}$. 
\end{proposition}
\begin{proof}   
(a)  Proposition~\ref{rollrightleft}  applied to $c = R(\epsilon_i)R(\theta_1)\cdots \widehat{R(\theta_i)}\cdots R(\theta_n) $ yields $\mu(\epsilon_i)\cdot \theta_j = 0$ if $i\ne j$.  
 By equation~(\ref{multiplereflections}),
$\epsilon_i = \theta_i +x_i$ for some $x_i \in \mbox{Span}\left(\{\theta_1, \dots , \theta_{i-1}\}\right)$
for $1\le i \le n$.   Hence
\[\mu(\epsilon_i)\cdot \theta_i 
= \mu(\theta_i+x_i)\cdot \theta_i 
= \mu(\theta_i)\cdot \theta_i+\mu(x_i)\cdot \theta_i = 1+0 = 1.\]
 (b)  This is analogous to part (a).  
 \\
(c) This follows from (b) and the identity $\epsilon_i = -c\tau_i $.
\end{proof}
\subsection{Geometry of polytopes}  A convex polytope $P$ is the convex hull of a finite set of points in $\mathbb R^n$. An (affine) hyperplane $H$ is {\em bounding} for $P$ if $H$ meets $P$ and $P$ is contained in one of the closed halfspaces determined by $H$. A {\em face} of $P$ is any intersection of $P$ with a bounding hyperplane. A {\em vertex} of $P$ is any face of dimension $0$ and a {\em facet} is any face whose dimension is one less than the dimension of $P$.  Each facet of an $n$-dimensional polytope, $P$, determines a unique hyperplane  which contains it and is called the {\em support} of the facet. See \cite{Ziegler} for background on the theory of polytopes.

Assume that $\{ u_0, u_1, \dots , u_n\} \subset \reals^n$ is affinely independent and put $v_i = u_i-u_{i-1}$ for $i=1,\dots,n$.  Thus, the only solution to the equations 
\[\sum_{i = 0}^{n}\lambda_i = 0 \ \mbox{and}\ \sum_{i = 0}^{n}\lambda_iu_i = 0\]
is $\lambda_i = 0$ for each $i = 0,1, \dots , n$, or, equivalently, $\{v_1, \dots , v_n\}$ 
is a basis for $\reals^n$.  
Let $\{\hat{v}_1, \dots , \hat{v}_n\}$ be the corresponding dual basis so that 
$v_i\cdot \hat{v}_j = \delta_{ij}$ for $1 \le i,j \le n$.
\begin{proposition}\label{hyperplaneequations}
Let $H$ be a bounding hyperplane  of the $n$-simplex spanned by 
$\{ u_0, u_1, \dots , u_n\}$. Then
\begin{itemize}
\item[(a)]  $H = \{x \in \reals^n: \hat{v}_1\cdot (x - u_1) = 0\}$ if $H$ contains 
$\{u_1, \dots , u_n\}$,
\item[(b)]  $H = \{x \in \reals^n: \hat{v}_n\cdot (x - u_0) = 0\}$ if $H$ contains 
$\{ u_0, \dots , u_{n-1}\}$,
\item[(c)]  $H = \{x \in \reals^n: (\hat{v}_i-\hat{v}_{i+1})\cdot (x - u_0) = 0\}$ if 
$H$ contains\\ $\{ u_0, u_1, \dots , u_{i-1}, u_{i+1},\dots , u_n\}$  for  some $1\le i \le n-1$.
\end{itemize}
\end{proposition}
\begin{proof} (c)  The bounding hyperplane  has normal orthogonal to 
$v_1, \dots , v_{i-1},$ $ v_i+v_{i+1}, 
v_{i+2}, \dots, v_n$ and contains $u_0$.  The vectors $\hat{v}_i$ and $\hat{v}_{i+1}$ are both orthogonal to each element of 
$\{v_1, \dots , v_{i-1}, v_{i+2}, \dots, v_n\}$ and $\hat{v}_i-\hat{v}_{i+1}$ is orthogonal to $v_i+v_{i+1}$.
The proofs of (a) and (b)  are similar.
\end{proof}

Recall that a \emph{geometric $p$-simplex} in $\reals^n$ 
is the 
convex hull of an affinely independent set of $p+1$ points. A {\em geometric simplicial complex} is a set of  geometric simplices which contains every face of each of its elements and such that the intersection of any two of its elements is a face of both. The \emph{vertices}  of a geometric simplicial complex are its $0$-simplices.

\begin{definition}
A {\em triangulation} of a convex polytope, $P$, is
a geometric simplicial complex whose vertex set is contained in the vertex set of $P$ and whose union of simplices is equal to $P$. 
\end{definition}

\subsection{The $W$-permutahedron and its boundary}
\begin{definition}
Our preferred version of the permutahedron, $P(W)$, is defined to be the convex hull of the set of points 
$\{w(v_0)\mid w \in W\}$.
\end{definition}
\begin{lemma}\label{negativity}
For each $w \in W$ and $1\le i \le n$, $\beta_i \cdot w(v_0) \le \beta_i\cdot v_0$, with equality if 
and only if $w \in W_i$.  
\end{lemma}
\begin{proof}
Choose an $S$-reduced expression  
$s_{i_1}s_{i_2}\cdots s_{i_k}$ for $w$
 and, for $1\le j \le k$, let $\phi_j = s_{i_1}s_{i_2}\cdots s_{i_{j-1}}\alpha_{i_j}$.  
 Then each $\phi_j$ is a positive root, $R(\phi_j) = s_{i_1}\cdots s_{i_{j-1}}s_{i_j}s_{i_{j-1}}\cdots s_{i_1}$ and $w = R(\phi_k)R(\phi_{k-1})\cdots R(\phi_1)$.  \\
Set $w_j = R(\phi_j)R(\phi_{j-1})\cdots R(\phi_1)$ for $j \le k$ and note that $w_j  = s_{i_1}s_{i_2}\cdots s_{i_j}$. 
Thus 
\begin{eqnarray*}
\beta_i\cdot w_j(v_0) &=& \beta_i\cdot R(\phi_j)w_{j-1}(v_0) \\
&=& \beta_i\cdot [w_{j-1}(v_0)-2(\phi_j \cdot w_{j-1}(v_0))\phi_j] \\
&=& [\beta_i\cdot w_{j-1}(v_0)]-2(\phi_j \cdot w_{j-1}(v_0)) [\beta_i\cdot \phi_j]\\
&\le &  \beta_i\cdot w_{j-1}(v_0).
\end{eqnarray*}
The last inequality follows since  $\phi_j\cdot w_{j-1}(v_0) = \alpha_{i_j}\cdot v_0= 1$ and $\beta_i\cdot \phi_j\ge 0$ 
since $\phi_j$ is a positive root.   
Thus $\beta_i\cdot w_j(v_0) \le \beta_i\cdot w_{j-1}(v_0)$, with equality if and only if $\beta_i\cdot \phi_j = 0$.  
Hence, by induction, $\beta_i\cdot w(v_0) \le \beta_i\cdot v_0$.  For the equality part, we note that $\beta_i\cdot w(v_0) = \beta_i\cdot v_0$ if and only if 
$\beta_i\cdot \phi_j = 0$ for each $j$.  
Thus $\beta_i\cdot w(v_0) = \beta_i\cdot v_0$ implies that each $R(\phi_j)$ belongs to $W_i$ 
and, hence, $w \in W_i$.   On the other hand, if $w \in W_i$, then  $w(v_0) = v_0+\sum_{j \ne i}a_j\alpha_j$, for some real numbers $a_j$ (by equation~(\ref{multiplereflections})), so that 
$\beta_i \cdot w(v_0) = \beta_i\cdot v_0$.  
\end{proof}
\if{false}
\begin{corollary}\label{negativecone}  For each $w \in W$,  
$w(v_0)-v_0 = a_1\alpha_1+ \dots +a_n \alpha_n$, 
where $a_i \le 0$ for $1\le i \le n$ and $a_i = 0$ if and only if $w\in W_i$.
\end{corollary} 
\begin{proof} This follows from Lemma~\ref{negativity} since $\{\beta_1, \dots , \beta_n\}$ is dual to 
$\{\alpha_1, \dots , \alpha_n\}$.
\end{proof} 
\fi
\begin{proposition}\label{permfacets}
(See also Lemma 7.3.3 of \cite{Davis}).  The faces of $P(W)$ are precisely the convex hulls 
of the sets $\{w(v_0)\mid w\in D\}$ where $D$ is a coset of a standard  parabolic subgroup of $W$.
\end{proposition} 
\begin{proof} 
Using the $W$ action on $P(W)$ and on cosets we can restrict to  the case of faces which contain $v_0$.  We recall from 
\cite{Ziegler}  that each face of a polytope is the convex hull of those vertices contained in that face. 
Denote the convex hull of the set  
$\{w(v_0)\mid w \in W_i\}$ by $P_i$.    By Lemma~\ref{negativity}, $P_i$ is  a facet of $P(W)$ with support given by $\beta_i\cdot (x - v_0) = 0$.   More generally, it follows from Lemma~\ref{negativity} that the intersection of any collection of $P_i$'s is  a face of $P(W)$ with support given by $u\cdot (x - v_0) = 0$, 
where $u$ is the sum of the corresponding $\beta_i$'s. 

For the opposite inclusion, suppose that $u = a_1\beta_1+ \dots 
+a_n \beta_n$ is a fixed nonzero vector and that 
$u \cdot (x - v_0) \le 0$ for all $x \in P(W)$.  
Then $0\ge u \cdot (s_i(v_0) - v_0) = -2a_i$ gives $a_i \ge 0$ for $1\le i \le n$.  
By Lemma~\ref{negativity}, whenever 
$a_i >0$, the equation $u \cdot(w(v_0) -v_0) = 0$ holds if and only if $w \in W_i$.  Thus the 
set of $w \in W$ satisfying $u \cdot(w(v_0) -v_0) = 0$ is the intersection of those $W_i$ 
for which $a_ i> 0$.  Hence the 
face of $P(W)$ defined by $u \cdot(x -v_0) = 0$ is the intersection of those $P_i$ for which $a_ i> 0$. 
\end{proof} 
\subsection{Two $K(\pi, 1)$'s for a type $W$ Artin group}
We use the notation $\Pi(a,b;m)$ for 
the length $m$ word $abab\cdots $, which starts with $a$  and  alternates  in $a$ and $b$.  
The type  $W$ Artin group, $B(W)$, is given by the presentation 
\begin{equation}\label{E-artin}B(W) = \langle x_1, \dots , x_n\mid \Pi(x_i,x_j;m_{i,j})=  
\Pi(x_j,x_i;m_{i,j})\ \mbox{for}\ i\ne j\rangle,
\end{equation}
where  $m_{i,j} $ is the order of $s_is_j$ in $W$, for $i \ne j$.

\begin{remark}\label{rem:salvetti-identification}
For each face $Q$ of $P(W)$  
there is a unique element, $w_Q\in W$,  with $w_Q(v_0)\in Q$ and satisfying
\[l_S(w_Q) \le l_S(w'), \ \mbox{for all}\ w' \in W \ \mbox{with}\ w'(v_0)\in Q.\]
Thus $w_Q$ is the shortest $S$-length representative of the 
left coset of the standard parabolic  subgroup corresponding to $Q$.
Let $Q$ and $Q'$ be two faces in the same $W$-orbit and let $U$ be the standard parabolic subgroup of $W$, provided by Proposition~\ref{permfacets},  for which 
$\{w_Qu(v_0)\mid u \in U\}$ and $\{w_{Q'}u(v_0)\mid u \in U\}$ are the vertices of $Q$ and $Q'$, respectively. 
 The orthogonal transformation $w_Qw_{Q'}^{-1}$ maps $Q'$ bijectively onto $Q$
since it maps the vertices of $Q'$ bijectively onto those of $Q$.
\end{remark}

\begin{definition}
Let $\qsal(W)$ denote the quotient space of $P(W)$ obtained by identifying any pair of faces $Q$ and $Q'$ in the same $W$-orbit via the homeomorphism $w_{Q}w_{Q'}^{-1}$ as in Remark~\ref{rem:salvetti-identification}. By \cite[Theorem 1.4]{Salv94}, $\qsal(W)$ is a classifying space for $B(W)$, and we call it the {\em Salvetti $K(\pi,1)$}.
\end{definition}
\begin{remark}\label{sal-fg}
Since the identifications respect the CW-structure given by the faces of $P(W)$, the images of faces of $P(W)$ give a CW-structure on $\qsal(W)$. 
The $2$-skeleton of $\qsal(W)$ is the presentation $2$-complex of the 
presentation in Equation~\ref{E-artin}.  Denote by $y_0$ the single vertex of $\qsal(W)$ corresponding to  the equivalence class of the vertex $v_0$ of $P(W)$ and, for $1\le i \le n$, denote by $e_i$ the loop in $\qsal(W)$ based at $y_0$ corresponding to the equivalence class of the edge 
with endpoints $v_0$ and $s_i(v_0)$ in $P(W)$ (which we orient away from $v_0$). 
Every $2$-dimensional face of $P(W)$ is 
identified with one 
 whose vertex set is 
$\{\langle s_i, s_j\rangle(v_0)\}$, for some $1\le i<j \le n$, by Proposition~\ref{permfacets}. 
Thus there is a $2$-cell in $\qsal(W)$ for every pair $s_i,s_j$ of distinct elements in $S$.  
The attaching map for this $2$-cell is given by $\alpha\circ \beta^{-1}$, where $\alpha$ reads the alternating sequence of $m_{i,j}$ oriented loops  in $e_i$ and $e_j$, starting with $e_i$, while $\beta$ 
reads the corresponding alternating sequence of $m_{i,j}$ oriented loops starting with $e_j$.  It follows  
that the function which maps the homotopy class of $e_i$ to $x_i$, for $i = 1, \dots , n$, induces an isomorphism of $\pi_1(\qsal(W),y_0)$ with $B(W)$.
\end{remark}

Recall that $L=\{w \in W : w \le_T c\}$ is the lattice of $W$-noncrossing partitions. The order complex, $\Delta(L)$, is  the abstract simplicial complex 
whose $p$-simplices are those $(p+1)$-tuples,
$(w_0, w_1, \dots , w_p)$, for which 
$w_0<_T w_1 <_T  \dots <_T w_p$ 
in $(L,\le_T)$. Let $\vert L \vert$ be any topological realization of $\Delta(L)$. 
\begin{definition}
We denote by $\qncp(W)$ the quotient of $\vert L\vert$ obtained by identifying the cell of $(w_0, w_1, \dots , w_p)$ with the cell of $(e, w_0^{-1}w_1, \dots, w_0^{-1}w_p)$.
We call it the {\em non-crossing partition $K(\pi, 1)$ for B(W)}. 
(The space $\qncp(W)$ is shown to be a $K(\pi, 1)$ for $B(W)$ in \cite{Bessis} and \cite{BWKayPie}.)
\end{definition}

\begin{remark}\label{ncp-fg}
The cell structure of $\vert L \vert$ induces a CW-structure on $\qncp(W)$ with a single vertex, $z$ say, and one $1$-cell $u_w$ for every $e<w\le c$. The cell $u_w$ is the loop in $\qncp(W)$ given as the quotient of the geometric realization of the edge $(e,w)$ in $\vert L\vert$, which we orient away from $e$. This orients the loop $u_w$. As proved in \cite{Bessis} and \cite{BWKayPie}, mapping $x_i$ to the homotopy class of $u_{s_i}$ induces an isomorphism $B(W)\simeq \pi_1(\qncp(W),z)$.  
\end{remark}

\begin{lemma}\label{iso-fg} 
Mapping the homotopy class of $e_i$ to the homotopy class of $u_{s_i}$ induces an isomorphism $\pi_1(\qsal(W),y_0) \simeq \pi_1(\qncp(W),z)$.
\end{lemma}
\begin{proof}
The isomorphism is obtained by composing the isomorphisms of Remarks~\ref{sal-fg}~and~\ref{ncp-fg}.
\end{proof}

\begin{remark}
It will be useful to us later (Theorem~\ref{htyequiv}) that the space $\qncp(W)$ has the structure of a $\Delta$-complex or trisp.  For background on  $\Delta$-complexes and trisps see, respectively, \cite[Section 2.1]{H} and \cite[Section 2.3]{K}.  The 
$\Delta$-complex structure here is a special case of the \emph{interval complex} structure of 
\cite[Definition~2.8]{PaoliniSalvetti}.
\end{remark}

\section{Special $n$-simplices}
\label{triangfacets}
\begin{proposition}\label{affineindependent} 
Assume that $w_0,\dots,w_n \in W$ and 
$e \lessdot_T w_0^{-1}w_1 \lessdot_T \cdots \lessdot_T w_0^{-1}w_n$ is a maximal 
chain in the non-crossing partition lattice. 
Then the set $\{w_0(v_0), w_1(v_0), \dots , w_n(v_0)\}$  is  affinely 
independent, thus forming the vertex set of an $n$-simplex in $P(W)$.
\end{proposition}
\begin{proof}  There exist unique positive, unit  roots $\theta_1, \dots , \theta_n$ so that 
$R(\theta_i) = w_{i-1}^{-1}w_i$, for $1 \le i \le n$, and $c = R(\theta_1)\cdots R(\theta_n)$.   
In \S~\ref{rootsduals}, we saw that $\{\epsilon_1, \dots , \epsilon_n\}$ is a basis for $\reals^n$.
As $\theta_{i+1}\cdot v_0>0$ for $0 \le i \le n-1$, and
\begin{eqnarray*}
w_{i+1}(v_0)-w_i(v_0) &=& w_iR(\theta_{i+1})(v_0)-w_i(v_0)\\
&=& -2(\theta_{i+1}\cdot v_0)w_i(\theta_{i+1})\\
&=& -2(\theta_{i+1}\cdot v_0)w_0(\epsilon_{i+1}),
\end{eqnarray*} 
the set $\{w_{1}(v_0)-w_0(v_0), w_{2}(v_0)-w_1(v_0), \dots , w_{n}(v_0)-w_{n-1}(v_0)\}$ 
is linearly  independent  and, hence,  $\{w_0(v_0), w_1(v_0), \dots , w_n(v_0)\}$ 
is affinely independent.  
\end{proof}

The $n$-simplices which appear in our triangulation of $P(W)$ satisfy a further restriction.

\begin{definition}\label{facets} 
Let $(w_0, w_1, \dots , w_n)$ be an ordered subset of $W$ such that
\begin{itemize}
	\item[(i)] $e \lessdot_T w_0^{-1}w_1 \lessdot_T \cdots \lessdot_T w_0^{-1}w_n=c$ is a maximal 
	chain in the non-crossing partition lattice, and 
	\item[(ii)]  
    $l_S(w_i) = l_S(w_0)+i$ for $1 \le i \le n$.
\end{itemize}
Then the simplex spanned by $\{w_0(v_0), w_1(v_0), \dots , w_n(v_0)\}$  
is called a \textbf{special $n$-simplex of $P(W)$} and 
is denoted by 
$\sns{w_0}{w_1}{w_n}$. 
\end{definition}
\begin{remark} 
\label{genquot}
Suppose $u\in W$ has the property that  $l_S(uc) = l_S(u)+n$. (The set of all such $u$ is the generalized quotient $W/\{c\}$, in the terminology of 
\cite{BjWa}.)   Then the convex hull of
\[\{u(v_0), us_1(v_0), us_1s_2(v_0), \dots ,  uc(v_0)\}\]
 forms a special $n$-simplex.   
Furthermore, if 
$\sns{w_0}{w_1}{w_n}$ is a special $n$-simplex with $w_0 = u$, then 
$(w_0, w_1,\dots,  w_n)$ is a maximal chain in the Bruhat interval from $u$ to $uc$ in $W$.   
\end{remark} 
\subsection{Characterising special $n$-simplices}
\label{facetschar}
In this subsection, 
we fix  $w_0 \in W$ and a factorisation $c = R(\theta_1)R(\theta_2)\cdots R(\theta_n)$ where 
 $\theta_1, \dots , \theta_n$  are  positive, unit roots.  
 For $1\le j \le n$, define $w_j \in W$ by
\[w_j = w_0R(\theta_1)R(\theta_2)\cdots R(\theta_j).\] 
As $\{\mu(\theta_1), \mu(\theta_2), \dots, \mu(\theta_n)\}$ and $\{\tau_1,\dots,\tau_n\}$
are dual bases of $\reals^n$, by part (b) of Proposition~\ref{dualbases}, we have 
\begin{equation}
\label{E-coeffs}
w_n^{-1}(v_0) = a_1 \mu(\theta_1)+ \dots + a_n\mu(\theta_n)
\end{equation}
with 
\begin{equation}
\label{E-coeffs2}
a_j = \tau_j \cdot w_n^{-1}(v_0)= w_n(\tau_j)\cdot v_0 \neq 0 \mbox{ \ for } j= 1,\dots,n.
\end{equation}
The final inequality holds since $w_n(\tau_j)$ is a root and $v_0$ does not lie on any  reflection hyperplane of $W$.

\begin{proposition}\label{basischange} 
For $1\le j \le n$, we have
\[w_j^{-1}(v_0) = a_1\mu(\theta_1)+ \dots + a_j\mu(\theta_j) 
+ a_{j+1}\mu(c \theta_{j+1}) +\dots +a_n\mu(c \theta_n).\]
\end{proposition}
\begin{proof} Use reverse induction on $j$. The case $j=n$ is true by equation~(\ref{E-coeffs}).
   Assume now that $j\le n$ and that
\[w_{j}^{-1}(v_0) = a_1\mu(\theta_1)+ \dots + a_{j}\mu(\theta_{j}) 
+ a_{j+1}\mu(c \theta_{j+1}) +\dots +a_n\mu(c \theta_n).\]
By part (b) of Proposition~\ref{rollrightleft}, $\theta_j$ is orthogonal to each of the vectors 
$\mu(\theta_1), \dots,  \mu(\theta_{j-1})$ and $\mu(c \theta_{j+1}), \dots,  \mu(c \theta_n)$.  Hence  
\begin{eqnarray*}
w_{j-1}^{-1}(v_0) &=& R(\theta_j)w_{j}^{-1}(v_0)\\
&=& R(\theta_j)[a_1\mu(\theta_1)+ \dots + a_{j-1}\mu(\theta_{j-1})]+ R(\theta_j)[a_{j}\mu(\theta_{j})]\\
&&+ R(\theta_j)[a_{j+1}\mu(c \theta_{j+1}) +\dots +a_n\mu(c \theta_n)] \\
&=& a_1\mu(\theta_1)+ \dots + a_{j-1}\mu(\theta_{j-1})+ R(\theta_j)[a_{j}\mu(\theta_{j})]\\
&&+ a_{j+1}\mu(c \theta_{j+1}) +\dots +a_n\mu(c \theta_n) \\
&=& a_1\mu(\theta_1)+ \dots + a_{j-1}\mu(\theta_{j-1})+ a_{j}\mu(c\theta_{j})\\
&&+ a_{j+1}\mu(c \theta_{j+1}) +\dots +a_n\mu(c \theta_n),
\end{eqnarray*}
since $R(\theta_j)\mu(\theta_{j}) = c\mu(\theta_{j}) = \mu(c\theta_{j})$.
\end{proof}

\begin{proposition}
\label{lengthen}
$l_S(w_{j}) > l_S(w_{j-1})$ if and only if $a_j < 0$.  
\end{proposition}
\begin{proof}  Since $w_j=w_{j-1}R(\theta_j)$ and $\theta_j$ is a positive root, 
$l_S(w_{j}) > l_S(w_{j-1})$ if and only if $w_j(\theta_j)$ is a negative
 root (by Proposition 5.7 of \cite{Hum90}),   hence  if and only if 
$w_j(\theta_j)\cdot v_0 <0$. But $w_j(\theta_j)\cdot v_0= w_n(\tau_j)\cdot v_0 = a_j$,
since $w_j(\theta_j)= w_n(\tau_j)$.
\end{proof}

\begin{proposition}\label{char}
The following conditions are equivalent.
\begin{itemize}
\item[(a)]  $\sns{w_0}{w_1}{w_n}$ is a special $n$-simplex.  
\item[(b)]  $w_0^{-1}(v_0) \in \mbox{cone}(\{-\mu(c\theta_1), 
-\mu(c \theta_2), \dots  , -\mu(c \theta_n)\})$. 
\item[(c)]  $w_n^{-1}(v_0) \in \mbox{cone}
(\{-\mu(\theta_1), -\mu( \theta_2), \dots  , -\mu( \theta_n)\})$. 
\end{itemize} 
\end{proposition} 
\begin{proof} 
The equivalence of (b) and (c) follows because $w_n^{-1}=c^{-1}w_0^{-1}$ and 
$\mu\circ c=c\circ \mu$.  
We prove that (a) is equivalent to (b).  
\vskip .1cm
(a) $\Rightarrow$ (b):  Assume that $\sns{w_0}{w_1}{w_n}$ is a special $n$-simplex
and write $w_0^{-1}(v_0) = a_1\mu(c \theta_1)+\dots + a_n\mu(c \theta_n)$.  
Since $l_S(w_{i}) = l_S(w_{i-1})+1$,  Proposition~\ref{lengthen}  gives $a_i <0$ for $1\le i\le n$.
 \vskip .1cm
(b) $\Rightarrow$ (a):  Assume that $w_0^{-1}(v_0) = a_1\mu(c \theta_1)+\dots + a_n\mu(c \theta_n)$ with $a_i <0$ for $1\le i\le n$.  Then $l_S(w_{i}) > l_S(w_{i-1})$ for $1\le i\le n$, by
Proposition~\ref{lengthen}, and hence $l_S(w_n)\ge l_S(w_0)+n$.
As  $l_S(w_n) = l_S(w_0c) \le l_S(w_0)+ l_S(c) = l_S(w_0)+  n$, we must have
$l_S(w_n) =  l_S(w_0)+n$ which now forces $l_S(w_{i}) = l_S(w_{i-1})+1 $ 
for $1 \le i \le n$. \end{proof} 

\begin{remark} In \cite{Lattices}, the authors construct a spherical simplicial complex, $X(c)$, 
which models the non-crossing partition lattice, $L$. This
complex gives a corresponding triangulation of the
simplicial cone generated by
 $\{\alpha_1, \dots , \alpha_n\}$ which has a top-dimensional, simplicial cone 
 generated by $\{\rho_{i_1}, \dots , \rho_{i_n}\}$ for each (decreasing) factorisation 
 $c = R(\rho_{i_n})\cdots R(\rho_{i_1})$ with $1\le i_1<i_2< \dots  <i_n\le nh/2$.

Suppose $\sns{w_0}{w_1}{w_n}$ is a special $n$-simplex with $w_i$ and $\theta_i$ as above.  Then 
$-\mu^{-1}[w_n^{-1}(v_0)]$ lies in the simplicial cone generated by 
$\{\alpha_1, \dots  , \alpha_n\}$.  By Corollary~7.5 of 
\cite{Lattices}, there exists  a unique sequence $1\le i_1<i_2< \dots  <i_n\le nh/2$ such that 
$c = R(\rho_{i_n})\cdots R(\rho_{i_1})$ and the simplicial cone generated by $\{\rho_{i_1}, \dots  , \rho_{i_n}\}$ contains the point $-\mu^{-1}[w_n^{-1}(v_0)]$. 
Thus, there is a unique special $n$-simplex with final vertex $w_n$ 
whose associated reflection sequence is $T$-decreasing.  
\end{remark}
\section{Intersections of special $n$-simplices with faces of the Permutahedron}
\label{permfacechar}
In this section 
we characterise each $p$-dimensional  intersection of any special $n$-simplex with a $p$-dimensional face of $P(W)$ as the convex hull of a set of vertices of the face which satisfies properties (i)--(iv) below.  This involves showing how to extend such a set 
of vertices to obtain the vertex set of a special 
$n$-simplex.  This process can  require several  modifications 
of an initial factorisation of $c$ (see proof of Proposition~\ref{extendboundarysimplex2}). 
Throughout this section we will use the 
facts that 
$\Pi_1 = \{\alpha_1, \dots , \alpha_k\}$ and 
$\Pi_2 = \{\alpha_{k+1}, \dots , \alpha_n\}$ are orthonormal sets and that the last $n$ positive roots are 
$-c^{-1}\alpha_{1}, \dots , -c^{-1}\alpha_k$ (possibly permuted)  followed by $ \alpha_{k+1}, \dots , \alpha_n$ (possibly permuted). 

Let $Q$ be a 
$p$-dimensional face of $P(W)$ and
assume that the intersection of $Q$ with 
some special $n$-simplex, 
$\sns{w_0}{w_1}{w_n}$, is  a $p$-dimensional face of
$\sns{w_0}{w_1}{w_n}$. 
By Proposition~\ref{permfacets}, $Q$  is the convex hull 
of the set of translates of $v_0$ by the elements of a coset of some standard  
parabolic subgroup $W'=\langle s_{i_1},\dots,s_{i_p} \rangle$, where $i_1<i_2<\cdots<i_p$.  
Since $l_S(w_i) < l_S(w_j)$ for 
$1\le i<j \le n$ and since   
$w_q^{-1}w_r \le_T s_{i_1}\cdots s_{i_p}$ whenever 
 $q<r$ and  $w_q(v_0), w_r(v_0) \in Q$, 
the $p+1$ vertices in $Q$ must (by the reasoning of the proof of Proposition~\ref{char}) be consecutive in $\sns{w_0}{w_1}{w_n}$.  Thus 
the special $n$-simplex must have the form $\langle w_0|\cdots |w_{j}|\cdots |w_{j+p}|\cdots  |w_n\rangle$
with 
\begin{itemize}
\item[(i)] the vertices $w_{j}(v_0), \dots , w_{j+p}(v_0)$ in $Q$,
\item[(ii)] $w_{j+p}=w_{j}s_{i_1}\cdots s_{i_p}$, 
\item[(iii)] $w_{j}\lessdot_T w_{j+1} \lessdot_T  \cdots \lessdot_T  w_{j+p}$  and 
\item[(iv)] $l_S(w_{i+1}) = l_S(w_i)+1$ for $i = j,  \dots , j+p-1$.
\end{itemize}
Proposition~\ref{extendboundarysimplex2} shows the converse, namely, that 
$(p+1)$ vertices of $Q$, $w_{j}(v_0)$, \dots, $w_{j+p}(v_0)$, which satisfy
conditions (i), (ii), (iii) and  (iv) 
comprise the vertex set of the intersection of $Q$ with some special $n$-simplex. 
\begin{lemma}\label{rollaround}
Assume  that $w \in W$ and $\theta_1,\dots , \theta_n$  are positive roots 
such that 
\(c= R(\theta_1)\cdots R(\theta_{n})\) and 
\(w^{-1}(v_0) = a_1\mu(c \theta_1)+ \dots +a_n\mu(c\theta_n)\).
 If $ \theta_n$ is one of the last $n$ positive roots
then the root $\rho := -c\theta_n$ is positive,
 and satisfies
$c = R(\rho)R(\theta_1)\cdots R(\theta_{n-1})$.   Furthermore,
\[ \left[wR(\rho)\right]^{-1}(v_0) = (-a_n)\mu(c\rho)+a_1\mu(c \theta_1)+ \dots +
a_{n-1}\mu(c\theta_{n-1}).\]
\end{lemma}
\begin{proof}  Note that $c = R(\rho)R(\theta_1)\cdots R(\theta_{n-1})$ by equation~(\ref{E-rollright}).  
Since $\theta_n$ is one of the last $n$ positive roots, $c\theta_n$ is a negative root and hence $\rho$ is 
positive.  Since $c= R(c\theta_1)\cdots R(c\theta_{n})$, by Proposition~\ref{rollrightleft} the root $c\theta_n$ (and hence $\rho$) is orthogonal to each of $\mu(c\theta_1),\dots,\mu(c\theta_{n-1})$. It follows that
\begin{eqnarray*}
\left[wR(\rho)\right]^{-1}(v_0)&=&R(\rho)\left[a_1\mu(c \theta_1)+ 
\dots + a_{n-1}\mu(c \theta_{n-1})+ a_n\mu(c \theta_n)\right]\\
&=&R(\rho)\left[a_1\mu(c \theta_1)+ \dots + a_{n-1}\mu(c \theta_{n-1})+ (-a_n)\mu(\rho)\right]\\
&=&a_1\mu( c\theta_1)+ \dots + a_{n-1}\mu(c\theta_{n-1})+ (-a_n)R(\rho)\mu(\rho)\\
&=&a_1\mu( c\theta_1)+ \dots + a_{n-1}\mu(c\theta_{n-1})+ (-a_n)c\mu( \rho)\\
&=&(-a_n)\mu(c\rho)+a_1\mu( c\theta_1)+ \dots + a_{n-1}\mu(c\theta_{n-1})
\end{eqnarray*}
since, as noted in \S\ref{NCPsection}, $R(\rho)\mu(\rho)=c\mu(\rho)$ and
$\mu\circ c= c \circ \mu$. 
\end{proof}  
\begin{proposition}\label{rollseveral}
Assume that $w_0 \in W$ and 
$\theta_1,\dots,\theta_{n}$ are positive unit roots such that
\(c = R(\theta_1)\cdots R(\theta_{n})\).  Assume further that 
$\theta_{l+1},\dots , \theta_n$ are among the last $n$ positive roots
and that 
\[ w_0^{-1}(v_0) = a_1\mu(c \theta_1)+ \dots +
a_{n}\mu(c\theta_{n})\]
with $a_1, \dots , a_l$ negative and  $a_{l+1}, \dots , a_n$ positive.   Then  
\[w_0(v_0),  w_1(v_0), \dots , w_{l-1}(v_0),  w_l(v_0)\]
are the vertices of a face of a special $n$-simplex, where  $w_i = w_0R(\theta_1)\cdots R(\theta_{i})$ for 
$i = 1, \dots , l$.
\end{proposition}
\begin{proof}  By repeated application of Lemma~\ref{rollaround}, $-c\theta_n, \dots , -c\theta_{l+1}$ 
are among the first $n$ positive roots,
\[c = R(c\theta_{l+1})\cdots R(c\theta_{n})R(\theta_1)\cdots R(\theta_{l})\]
and  
\begin{eqnarray*}
\left(w_0R(c\theta_n)\cdots R(c\theta_{l+1})\right)^{-1}(v_0) &=& (-a_{l+1})\mu(-c^2\theta_{l+1})+\dots +(-a_n)\mu(-c^2\theta_n)\\
&&+a_1\mu(c \theta_1)+ \dots +
a_{l}\mu(c\theta_{l}).
\end{eqnarray*}
By Proposition~\ref{char}~(b),  there is a special $n$-simplex whose vertices are 
$w_{l-n}(v_0), \dots , w_{-1}(v_0),w_{0}(v_0), \dots , w_{l}(v_0)$, where 
\[w_{-k} = w_0R(c\theta_n)R(c\theta_{n-1})\cdots R(c\theta_{n-k+1}).\]
\end{proof} 
We first extend sets of $n-1$ vertices satisfying (i)--(iv) to special $n$-simplices.  This special case is used in Theorem~\ref{triangulation}.
\begin{proposition}\label{extendtopboundarysimplex}
Fix $i \in \{1,\dots,n\}$ and let $c'$ be the  
standard parabolic Coxeter element given by 
$c' =  s_1\cdots s_{i-1}s_{i+1}\cdots s_n$. Assume that $w_0 \in W$ and 
$\theta_1,\dots,\theta_{n-1}$ are positive unit roots such that
\(c' = R(\theta_1)\cdots R(\theta_{n-1})\).  Assume further that $l_S(w_j) = l_S(w_{j-1})+1$
for $1 \le j \le n-1$, where  $w_j = w_0R(\theta_1)\cdots R(\theta_j)$.
Then $w_0(v_0), w_1(v_0), \dots , w_{n-1}(v_0)$ are the vertices of an $(n-1)$-dimensional face of a  unique special $n$-simplex.   
\end{proposition}
\begin{proof}  From subsection~\ref{rootsduals}, the root $\tau_{i}= s_ns_{n-1}\cdots s_{i+1}\alpha_i$ is positive and  
\[
	c = s_1\cdots \widehat{s_{i}}\cdots s_nR(\tau_{i})
	= R(\theta_1)\cdots R(\theta_{n-1})R(\tau_{i}).
\]
It follows (using Equation~\ref{E-coeffs2}) that  
there exist  unique, non-zero scalars, $a_1, \dots , a_n$, for which 
\[w_0^{-1}(v_0) = a_1\mu(c \theta_1)+ \dots +a_{n-1}\mu(c \theta_{n-1})+ a_n\mu(c\tau_i).\]
Our hypotheses, together with Proposition~\ref{lengthen},  imply that
 $a_i <0$ for $1\le i\le n-1$.
 If $a_n< 0$,  then Proposition~\ref{char} implies that  
\[
\langle w_0|w_1| \dots | w_{n-1}|w_{n-1}R(\tau_i)\rangle\] 
is a special $n$-simplex, as required.

For the case $a_n> 0$,  we note that $\tau_i$ is among the last $n$ positive roots by 
Example~\ref{firstnlastn}.  Thus Proposition~\ref{rollseveral} with $l=n-1$ can be applied to give the required special $n$-simplex.

For the uniqueness, first note that  
since $w_0(v_0), w_1(v_0), \dots , w_{n-1}(v_0)$ must be consecutive in any special $n$-simplex containing them, the only possible such $n$-simplices are of the form 
\[\langle w_0|w_1| \dots | w_{n-1}|w_{n-1}R(\tau_i)\rangle \quad \mbox{or} \quad \langle  w_0R(\epsilon_i)|w_0|w_1| \dots | w_{n-1}\rangle.\]
The first arises when the coefficient $a_n$ above is negative; the second arises if $a_n$ is positive since $\epsilon_i = -c\tau_i$.
Since $a_n \ne 0$, exactly one of these possibilities must occur.
\end{proof}  
The following technical result about linear combinations of positive roots is used in the proof of  Proposition \ref{extendboundarysimplex2}.
\begin{lemma}
\label{newlemma}
Assume that $1 \le q_1<q_2<\cdots < q_t \le k$, that $k+1 \le p_1<\cdots p_r\le n$
and that $d_1,\dots,d_t,b_1,\dots,b_r \in \reals$. Let $w$ be the orthogonal transformation $w=s_{p_1}\cdots s_{p_r}$ and let $v$ be the vector given by
\[ v= b_1\alpha_{p_1}+\cdots+b_r\alpha_{p_r}+d_1\bigl(-c^{-1}\alpha_{q_1}\bigr)
+ \cdots + d_t\bigl(-c^{-1}\alpha_{q_t}\bigr).\]
Then $-wc^{-1}(\alpha_{q_1}),\dots, -wc^{-1}(\alpha_{q_t})$ are positive roots and 
\[  v= d_1\bigl(-wc^{-1}\alpha_{q_1}\bigr)
+ \cdots + d_t\bigl(-wc^{-1}\alpha_{q_t}\bigr)+
b'_1\alpha_{p_1}+\cdots+b'_r\alpha_{p_r}\]
for some $b'_1,\dots,b'_r \in \reals$. In particular, the  same coefficients $d_1, \dots , d_t$ appear in both of the above expressions for $v$.
\end{lemma}

\begin{proof} 
Fix $p\in \{ p_1,\dots,p_r\}$ and $q \in \{q_1,\dots,q_t\}$.   
Since $s_p$ permutes the set $\Phi^+\setminus\{\alpha_p\}$ (by Proposition~1.4 of \cite{Hum90})
and  fixes  $\alpha_{k+1},\dots,\alpha_{p-1},\alpha_{p+1},\dots,\alpha_n$ (by orthogonality
of $\Pi_2$),  $s_p$ must permute $\Phi^+\setminus \Pi_2$. It follows that $w$ permutes
$\Phi^+\setminus \Pi_2$. Since $-c^{-1}\alpha_q$ is one of the 
last $n$ positive roots but not one of the last $n-k$ positive roots, we have 
$-c^{-1}\alpha_q \in \Phi^+\setminus \Pi_2$ and, hence, $w(-c^{-1}\alpha_q)\in \Phi^+\setminus \Pi_2$ is a positive root.

For the second assertion, for $u \in \reals^n$, equation~(\ref{multiplereflections}) gives 
\[ wu = u+D_{1}\alpha_{p_1}+\cdots+
D_{r}\alpha_{p_r},\]
for some $D_{1}, \dots , D_{r} \in \reals$, so that 
\[u = wu -D_{1}\alpha_{p_1}-\cdots-
D_{r}\alpha_{p_r}.\]
Applying this to each $u \in \{-c^{-1}\alpha_{q_1},\dots,-c^{-1}\alpha_{q_r}\}$ in 
the original expression for $v$ and rearranging yields an expression
for $v$ of the form in the second assertion.
\end{proof}

\begin{proposition}\label{extendboundarysimplex2}
Suppose $c' =  s_1\cdots \widehat{s_{i_1}}\cdots \widehat{s_{i_r}}\cdots s_n$ is a standard parabolic Coxeter element obtained by deleting $r$ simple reflections from $c$ and we have a factorisation 
$c' = R(\theta_1)\cdots R(\theta_{n-r})$,
with each $\theta_i$ a positive root.  Suppose further that $w_0 \in W$ and 
that $l_S(w_j) = l_S(w_{j-1})+1$, where\\ $w_j := w_0R(\theta_1)\cdots R(\theta_j)$ for 
$1\le j \le n-r$.
Then  $w_0(v_0), w_1(v_0)$, \dots, $w_{n-r}(v_0)$ are the vertices of an 
$(n-r)$-dimensional face of a special $n$-simplex.   
\end{proposition}
\begin{proof}  Our proof starts with an initial factorisation of $c$ as a product of reflections.
We then make as many as three controlled modifications to this factorisation to obtain a factorisation of $c$ to which 
Proposition~\ref{rollseveral} may be applied. For the purpose of this proof, the exact order within
each of the commuting sets of reflections, 
$\{s_1,\dots,s_k\}$ and $\{s_{k+1}, \dots , s_n\}$, is unimportant and we will reorder at various points to 
improve readability.

\textbf{Initial Factorisation:}  We can assume, after reordering,
that there are integers $j$ and $l$ with $1\le j\le k$ and $k+1 \le l\le n$ with 
$\{s_{i_1}, \dots , s_{i_r}\} = \{s_1, \dots ,s_j, s_{l+1}, \dots s_n\}$ so that
\[c' = s_{j+1}\cdots s_ks_{k+1}\cdots  s_l.\] 
(If either of the sets $\{i_1, \dots , i_r\} \cap \{1,\dots ,k\}$ or $\{i_1, \dots , i_r\} \cap \{k+1,\dots ,n\}$ is empty, the proof simplifies.)
 We define $c_L = s_{1}\cdots s_{j}$ and $c_R = s_{l+1}\cdots s_{n}$, and note 
 that $c = c_Lc'c_R= c'c_Rc^{-1}(c_L)c$.    Thus 
 \[c = R(\theta_1)\cdots R(\theta_{n-r})R(\alpha_{l+1})\cdots R(\alpha_{n})R(-c^{-1}\alpha_{1})\cdots R(-c^{-1}\alpha_{j}),\]
where the roots $\theta_1, \dots, \theta_{n-r}, \alpha_{l+1}, \dots, \alpha_{n}, -c^{-1}\alpha_{1}, \dots ,-c^{-1}\alpha_{j}$ are all positive.  
By Equation~\ref{E-coeffs2}, there are  nonzero scalars $a_{1}, \dots , a_{n-r}$, $b_{l+1}, \dots , b_{n}$ and  $d_{1}, \dots , d_{j}$ with
\begin{eqnarray*}
(\mu c)^{-1}w_0^{-1}(v_0)  &=& a_{1}\theta_1+\dots +a_{n-r}\theta_{n-r}\\
&&+b_{l+1}\alpha_{l+1}+\dots+ b_{n}\alpha_{n} \\
&&+d_{1}(-c^{-1} \alpha_{1})+\dots +d_{j}(-c^{-1} \alpha_{j}).
\end{eqnarray*}
As $l_S(w_p) = l_S(w_{p-1})+1$, for $1\le p \le n-r$, Proposition~\ref{lengthen} implies that $a_i$ 
is negative for $1\le i \le n-r$.  

\textbf{Second Factorisation:} 
Since $\{s_1, \dots, s_j\}$ and $\{s_{l+1}, \dots, s_n\}$ are commuting sets, we can reorder so that 
there are integers 
$m\in \{l+1, \dots , n\}$ and $i\in \{1, \dots , j\}$ such that 
$b_{l+1}, \dots , b_{m},d_{1}, \dots , d_{i}$ are negative and 
$b_{m+1}, \dots , b_{n}$, $d_{i+1}, \dots , d_{j} $ are positive.  
(If $b_{l+1}, \dots , b_{n}$ or $d_{1}, \dots , d_{j}$ all have the same sign, the proof simplifies.)
  Thus $c = c'c_1c_2c_3c_4$, where 
 \[c_1 = R(\alpha_{l+1})\cdots R(\alpha_{m}), \quad c_2 = 
 R(\alpha_{m+1})\cdots R(\alpha_n),\]
 \[c_3 = R(-c^{-1}\alpha_{1})\cdots R(-c^{-1}\alpha_{i}), 
 \quad c_4 = R(-c^{-1}\alpha_{i+1})\cdots R(-c^{-1}\alpha_{j}),\]  
 and 
\begin{eqnarray*}
(\mu c)^{-1}w_0^{-1}(v_0)  &=& a_{1}\theta_1+\cdots +a_{n-r}\theta_{n-r}\\
&&+b_{l+1}\alpha_{l+1}+\cdots+ b_{m}\alpha_{m} \\
&&+b_{m+1}\alpha_{m+1}+\cdots+ b_{n}\alpha_{n} \\
&&+d_{1}(-c^{-1} \alpha_{1})+\cdots +d_{i}(-c^{-1} \alpha_{i})\\
&&+d_{i+1}(-c^{-1} \alpha_{i+1})+\cdots +d_{j}(-c^{-1} \alpha_{j}).
\end{eqnarray*}
\textbf{Third Factorisation:}  Applying Lemma \ref{newlemma} to the sum of the third and fourth rows
of this expression yields  $b'_{m+1},\dots,b'_{n} \in \reals$ such that
\begin{eqnarray*}
(\mu c)^{-1}w_0^{-1}(v_0)  &=& a_{1}\theta_1+\cdots +a_{n-r}\theta_{n-r}\\
&&+b_{l+1}\alpha_{l+1}+\cdots+ b_{m}\alpha_{m} \\
&&+d_{1}(-c_2c^{-1} \alpha_{1})+\cdots +d_{i}(-c_2c^{-1} \alpha_{i})\\
&&+b'_{m+1}\alpha_{m+1}+\cdots+ b'_{n}\alpha_{n} \\
&&+d_{i+1}(-c^{-1} \alpha_{i+1})+\cdots +d_{j}(-c^{-1} \alpha_{j}),
\end{eqnarray*}
with all of the roots, including $-c_2c^{-1} \alpha_{1},\dots,-c_2c^{-1} \alpha_{i}$,
positive. Note that the coefficients in the first, second and third rows are negative, the coefficients 
in the fifth row are positive but the signs of $b'_{m+1}, \dots , b'_n$ are unknown. 
Since $c_2c_3= (c_2c_3c_2^{-1})c_2$, we also have 
\[c= c'c_1R(-c_2c^{-1}\alpha_{1})\cdots R(-c_2c^{-1}\alpha_{i})c_2c_4.\]
 \vskip .2cm
\textbf{Fourth Factorisation:}  A final reordering of $\alpha_{m+1}, \dots , \alpha_n$ allows us to assume there exists
$t \in \{m+1, \dots , n\}$ with the property that $b'_{m+1}, \dots , b'_t$ are negative while 
$b'_{t+1}, \dots , b'_n$ are positive.   (If $b'_{m+1}, \dots , b'_n$ all have the same sign, the proof simplifies.) We now have 
\begin{eqnarray*}
(\mu c)^{-1}w_0^{-1}(v_0)  &=& a_{1}\theta_1+\dots +a_{n-r}\theta_{n-r}\\
&&+b_{l+1}\alpha_{l+1}+\dots+ b_{m}\alpha_{m} \\
&&+d_{1}(-c_2c^{-1} \alpha_{1})+\dots +d_{i}(-c_2c^{-1} \alpha_{i})\\
&&+b'_{m+1}\alpha_{m+1}+\dots+ b'_{t}\alpha_{t} \\
&&+b'_{t+1}\alpha_{t+1}+\dots+ b'_{n}\alpha_{n} \\
&&+d_{i+1}(-c^{-1} \alpha_{i+1})+\dots +d_{j}(-c^{-1} \alpha_{j}),
\end{eqnarray*}
where the coefficients in the first, second, third and fourth rows are negative while the coefficients 
in fifth and sixth rows are positive.  

\textbf{Final Step:}  Since the roots $\alpha_{t+1},\dots, \alpha_{n}$ and 
$-c^{-1} \alpha_{i+1}, \dots , -c^{-1} \alpha_{j}$ are among the last $n$ positive roots (by Section \ref{NCPsection}), 
we can apply Proposition~\ref{rollseveral} 
to produce a special $n$-simplex 
having $w_0(v_0), w_1(v_0), \dots , w_{n-r}(v_0)$ among its vertices.
\end{proof}
\begin{remark}\label{oneskeleton}
Applying Proposition~\ref{extendboundarysimplex2} in the cases $r = n$ and 
$r=n-1$ yields, for each vertex or edge of the permutahedron, a special $n$-simplex containing that vertex or edge.  
\end{remark}

\section{Rank 2 example}
\label{dihedral}
\begin{example}\label{dihedralex}
Consider the case where $W = C_2$, the symmetry group of the square.  Using the usual action on $\reals^2$, we fix a fundamental chamber 
\[C = \{(x,y)\mid x-y \ge 0, y \ge 0\},\]
bounded by the reflection lines $y = 0$ and $y = x$.  To avoid surds, we drop the unit length restriction and use the simple roots 
$\{\alpha_1 = (0,1),  \alpha_2 = (1,-1)\}$.  
The corresponding dual basis vectors are then $\{\beta_1 = (1,1),  \beta_2 = (1, 0)\}$, 
so that $v_0 = (2,1)$.  
The ordered sets $\{\rho_1, \dots , \rho_8\}$ and $\{\mu_1, \dots , \mu_8\}$ 
(as in section \ref{NCPsection}) are given by
\begin{center}
\begin{tabular}{|c|c|c|c|c|c|c|c|c|}
\hline
$i$&$1$&$2$&$3$&$4$&$5$&$6$&$7$&$8$\\
\hline
$\rho_i$&$(0,1)$&$(1,1)$&$(1,0)$&$(1,-1)$
&$(0,-1)$&$(-1,-1)$&$(-1,0)$&$(-1,1)$\\
\hline
$\mu_i$&$(1,1)$&$(1, 0)$&$(1,-1)$&$(0,-1)$
&$(-1,-1)$&$(-1,0)$&$(-1,1)$&$(0,1)$\\
\hline
\end{tabular}
\end{center} 
Putting $R_i = R(\rho_i)$ for $1\le i \le 4$, we have the four expressions for the Coxeter element 
$c= R_1R_4 = R_4R_3 =R_3R_2 =R_2R_1$.
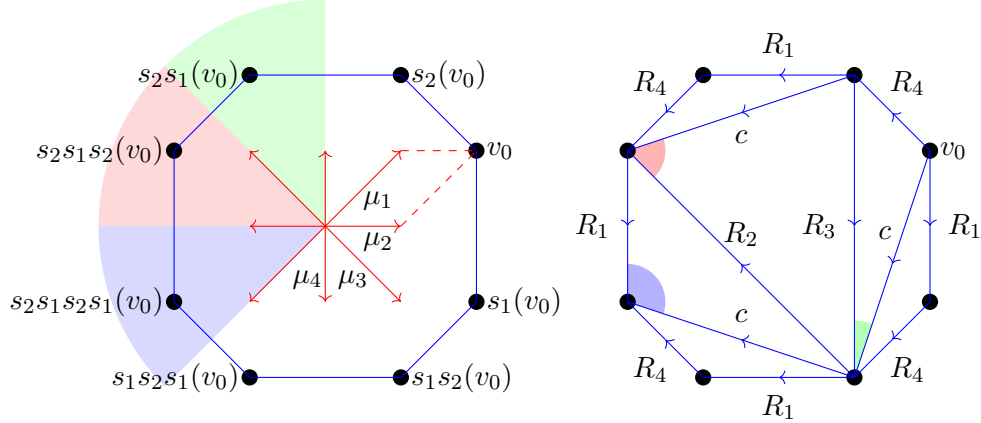
\begin{figure}[htbp]
\label{fig:rank2}
\begin{center}
\begin{tikzpicture}
 \tikzstyle{edge}=[->,thick]
 \path[fill=green!30,opacity=0.5] (0,0) --  (90:3) arc(90:135:3) -- cycle;
\path[fill=red!30,opacity=0.5] (0,0) --  (135:3) arc(135:180:3) -- cycle;
\path[fill=blue!30,opacity=0.5] (0,0) --  (180:3) arc(180:225:3) -- cycle;
\path[fill=green!30,shift = {(7,-2)}] (0,0) --  (71.5:0.75) arc(71.5:90:0.75) -- cycle;
\path[fill=red!30,shift = {(4,1)}] (0,0) --  (-45:0.5) arc(-45:18.4:0.5) -- cycle;
\path[fill=blue!30,shift = {(4,-1)}] (0,0) --  (-18.4:0.5) arc(-18.4:90:0.5) -- cycle;
 \draw[fill=black] (2 ,1) circle (0.1) node [right]{$v_0$};
 \draw[fill=black] (2 ,-1) circle (0.1) node [right]{$s_1(v_0)$};
 \draw[fill=black] (1,-2 ) circle (0.1) node [right]{$s_1s_2(v_0)$};
 \draw[fill=black] (-1,-2 ) circle (0.1) node [left]{$s_1s_2s_1(v_0)$};
 \draw[fill=black] (1,2 ) circle (0.1) node [right]{$s_2(v_0)$};
 \draw[fill=black] (-1,2) circle (0.1) node [left]{$s_2s_1(v_0)$};
\draw[fill=black] (-2,1 ) circle (0.1) node [left]{$s_2s_1s_2(v_0)$};
 \draw[fill=black] (-2,-1 ) circle (0.1) node [left]{$s_2s_1s_2s_1(v_0)$};
\draw [->,red](0,0)--(1,0);
\node at (0.7,0) [below= -0.05 cm]{$\mu_2$};
\draw [-,blue](2,1)--(1,2);
\draw [-,red,dashed](1,1)--(2,1);
\draw [-,red,dashed](1,0)--(2,1);
\draw [-,blue](2,1)--(2,-1);
\draw [-,blue](2,-1)--(1,-2);
\draw [-,blue](1,-2)--(-1,-2);
\draw [-,blue](-1,-2)--(-2,-1);
\draw [-,blue](-2,-1)--(-2,1);
\draw [-,blue](-2,1)--(-1,2);
\draw [-,blue](-1,2)--(1,2);
\draw [->,red](0,0)--(1,1);
\node at (0.7,0.7) [below = 0.1 cm] {$\mu_1$};
\draw [->,red](0,0)--(0,1);
\draw [->,red](0,0)--(-1,1);
\draw [->,red](0,0)--(-1,0);
\draw [->,red](0,0)--(-1,-1);
\draw [->,red](0,0)--(0,-1);
\node at (0,-0.7) [left=-0.1] {$\mu_4$};
\draw [->,red](0,0)--(1,-1);
\node at (0.7,-0.7) [left] {$\mu_3$};
 \draw[fill=black] (8 ,1) circle (0.1)node [right]{$v_0$};
 \draw[fill=black] (8 ,-1) circle (0.1) ;
 \draw[fill=black] (7,-2 ) circle (0.1) ;
 \draw[fill=black] (5,-2 ) circle (0.1) ;
 \draw[fill=black] (7,2 ) circle (0.1) ;
 \draw[fill=black] (5,2) circle (0.1) ;
\draw[fill=black] (4,1 ) circle (0.1) ;
 \draw[fill=black] (4,-1 ) circle (0.1) ;
\draw [->,blue](8,1)--(7.5,1.5);
\draw [-,blue](7.5,1.5)--(7,2);
\draw [->,blue](8,1)--(8,0);
\draw [-,blue](8,0)--(8,-1);
\draw [->,blue](8,-1)--(7.5,-1.5);
\draw [-,blue](7.5,-1.5)--(7,-2);
\draw [->,blue](7,-2)--(6,-2);
\draw [-,blue](6,-2)--(5,-2);
\draw [->,blue](5,-2)--(4.5,-1.5);
\draw [-,blue](4.5,-1.5)--(4,-1);
\draw [->,blue](4,1)--(4,0);
\draw [-,blue](4,0)--(4,-1);
\draw [->,blue](5,2)--(4.5,1.5);
\draw [-,blue](4.5,1.5)--(4,1);
\draw [->,blue](7,2)--(6,2);
\draw [-,blue](6,2)--(5,2);
\draw [->,blue](8,1)--(7.5,-0.5);
\draw [-,blue](7.5,-0.5)--(7,-2);

\draw [->,blue](7,-2)--(5.5,-1.5);
\draw [-,blue](5.5,-1.5)--(4,-1);
\draw [->,blue](7,2)--(5.5,1.5);
\draw [-,blue](5.5,1.5)--(4,1);
\draw [->,blue](7,2)--(7,0);
\draw [-,blue](7,0)--(7,-2);
\draw [->,blue](7,-2)--(5.5,-0.5);
\draw [-,blue](5.5,-0.5)--(4,1);
\node at (8,0) [right = 0.1 cm] {$R_1$};
\node at (6,-2) [below = 0.1 cm] {$R_1$};
\node at (6,2) [above = 0.1 cm] {$R_1$};
\node at (7.7,-1.5) [below = 0.1 cm] {$R_4$};
\node at (4.3,-1.5) [below = 0.1 cm] {$R_4$};
\node at (4.3,1.5) [above = 0.1 cm] {$R_4$};
\node at (7.7,1.5) [above = 0.1 cm] {$R_4$};
\node at (4,0) [left = 0.1 cm] {$R_1$};
\node at (7,0) [left = 0.1 cm] {$R_3$};
\node at (5.5,-0.5) [above = 0.1 cm] {$R_2$};
\node at (5.5,1.5) [below = 0.1 cm] {$c$};
\node at (5.5,-1.5) [above = 0.1 cm] {$c$};
\node at (7.4,-0.5) [above = 0.2 cm] {$c$};

\end{tikzpicture}
\end{center}
\caption{The permutahedron (which is slightly irregular because of the lengths 
of the simple roots) is shown on the left while the triangulation is 
shown on the right.   In the figure on the right, edge labellings follow the Cayley graph conventions.  
Thus a vertex $w_1(v_0)$ is connected to a vertex $w_2(v_0)$ by an edge labelled $u$ 
if $w_2 = w_1u$.
}
\end{figure}

For the factorisation $c = R_4R_3 = R(\rho_4)R(\rho_3)$, we see the single permutahedron vertex $s_2s_1(v_0)$ in the green sector denoting the cone on $\{-\mu_4, -\mu_3\}$ in Figure~\ref{fig:rank2} (left).  Thus there is a single special $2$-simplex labelled $c = R_4R_3$ in the triangulation in Figure~\ref{fig:rank2} (right).  Its last vertex is at $(s_2s_1)^{-1}(v_0) = s_1s_2(v_0)$ which is also indicated with a green sector.  Similar reasoning applies to the factorisations $c = R_3R_2$ and 
$c = R_2R_1$ using the colours red and blue respectively.  The union of the three coloured sectors in Figure~\ref{fig:rank2} (left) is the cone on $\{-\mu_1, -\mu_4\}$ giving three special $2$-simplices labelled $c = R_1R_4$ in the triangulation with endpoints at the corresponding vertices.  

To illustrate Proposition~\ref{extendboundarysimplex2} we consider the case of the $0$-cell with coordinates $(-2,1) = w(v_0)$ for $w = s_2s_1s_2\,(=w^{-1})$.  Here $k = 0$ and $c' = e$ in the notation of Proposition~\ref{extendboundarysimplex2}.   
Thus in the proof of the proposition we would have $c_L = R_1$ and 
$c_R  = R_4$, giving the initial expression $c = c'c_Rc^{-1}c_Lc = R_4R_3$.  
\[w^{-1}(v_0) = (-2,1) = (3)\mu_6 + (-1)\mu_5 = (3)c\mu_4 + (-1)c\mu_3.\]
So $c_2 = R_4$ and $c_3 = R_3$ ($c_1=e$ and $c_4= e$ here)  which gives 
$c_2c_3c_2^{-1} = R_1$.  We compute 
\[w^{-1}(v_0) = (-2,1) = (-1)\mu_3 + (1)\mu_6 = (-1)c\mu_1 + (1)c\mu_4.\]
 We now apply Lemma~\ref{rollaround}, using $c\rho_4 = -\rho_2$,  to get 
\[(-1, 2) = R_2w^{-1}(v_0) = R_2(-2, 1) = (-1)\mu_4+(-1)\mu_3 = (-1)c\mu_2+(-1)c\mu_1.\]
Since the coefficients are negative, part (c) of Proposition~\ref{char} gives a 
special $2$-simplex labelled $c = R_2R_1$ on the vertices $(1, -2) = (R_2w^{-1})^{-1}(v_0)$, 
$(-2,1) = w(v_0)$ and $(-2,-1) = (R_1w^{-1})^{-1}(v_0)$.
\end{example}
\section{Facets of special $n$-simplices}
\label{panelschar}
In this section 
we  determine 
the incidences of facets of special $n$-simplices
with other   special $n$-simplices
and with the boundary of the permutahedron.  
Throughout the section, $F = \sns{w_0}{w_1}{w_n}$ is a fixed 
special $n$-simplex. Thus,  $w_0,\dots,w_n \in W$ and there exist positive unit roots
$\theta_1, \dots , \theta_n$ with
$w_j = w_0R(\theta_1)R(\theta_2)\cdots R(\theta_j)$ and 
$l_S(w_j) = l_S(w_0)+j$ for $ 1\le j \le n$,
and 
 $c = R(\theta_1)R(\theta_2)\cdots R(\theta_n)$.
\begin{definition}\label{panels}
For each $i = 0, 1,\dots,n$, the $i$th facet of $F$ is  obtained by removing the vertex $w_i(v_0)$
from $F$ and is denoted by  $F_i$.  
\end{definition}

\begin{proposition}
\label{supports}
\begin{itemize}
\item[(i)] The support of the facet $F_0$ is the hyperplane 
through $w_1(v_0)$ and orthogonal to  $ w_0\mu(-c \theta_1) $.
\item[(ii)]    The support of the facet $F_n$ is the hyperplane through  $w_0(v_0)$
and orthogonal to $ w_0\mu(-c \theta_n)$.
\item[(iii)]     The support of the facet $F_i$ is the hyperplane through 
$w_0(v_0)$ and orthogonal to $ [(\theta_i\cdot v_0)w_0\mu(-c \theta_{i+1})-(\theta_{i+1}\cdot v_0)w_0\mu(-c \theta_{i})] $, for $1\le i \le n-1$.
\end{itemize}
\end{proposition} 
\begin{proof} 
Since $w_{i+1}(v_0)-w_{i}(v_0) =-2(\theta_{i+1}\cdot v_0)w_0(\epsilon_{i+1})$ for
$i=0,\dots,n-1$ (as in the proof of Proposition~\ref{affineindependent}), the dual basis
to
\[ \{w_1(v_0)-w_0(v_0), w_2(v_0)-w_1(v_0), \dots,w_n(v_0)-w_{n-1}(v_0)\} \]
is 
\[\left\{ \frac{w_0\mu(c\theta_1)}{2(\theta_{1}\cdot v_0)}, \dots , \frac{w_0\mu(c\theta_n)}{2(\theta_{n}\cdot v_0)} \right\}, \]
by  Proposition~\ref{dualbases}. 
The proof is completed by applying Proposition~\ref{hyperplaneequations}.
\end{proof}

\begin{proposition}\label{topface}
(i) If $\theta_1$ is not one of the first $n$ positive roots,
 then $F_0$ is  a facet of exactly one other special $n$-simplex, namely
\[\langle w_1|\cdots |w_n|w_nR(c^{-1}(\theta_1))\rangle, \] and the 
support of $F_0$ separates $w_0(v_0)$ 
from $w_nR(c^{-1}(\theta_1))(v_0)$.

\vskip 0.2cm
(ii)  If $\theta_1$ is one of the first $n$ positive roots then 
$F_0$ lies in a facet of $P(W)$.
\end{proposition} 

\begin{proof} (i)  Since  $l_S(w_i) = l_S(w_1)+i-1$ for $2\le i \le n$, condition (ii) of Definition~\ref{facets} implies that the only possible special $n$-simplices
containing $F_0$ must have the form
\[F'=\sns{w_0'}{w_1}{w_n} \mbox{ \  for some } w_0' \in W, \]
or
\[F''=\langle w_1|\cdots |w_n|w_{n+1}\rangle  \mbox{ \  for some } w_{n+1}\in W.\]
In the first case, the requirement that $(w_0')^{-1}w_n=c$ implies that $w_0'=w_nc^{-1}=w_0$
and, hence, $F'=F$. 
In the second case, the requirement  that $c=w_1^{-1}w_{n+1}$ (for a special $n$-simplex)
gives
\begin{eqnarray*}
c&=&w_1^{-1}w_{n}(w_n^{-1}w_{n+1})\\
&=& R(\theta_2)\cdots R(\theta_n)(w_n^{-1}w_{n+1}).
\end{eqnarray*}
Since $c = R(\theta_2)\cdots R(\theta_n)R(c^{-1}\theta_1)$, it follows that
$w_n^{-1}w_{n+1}=R(c^{-1}\theta_1)$ and, hence, $F''$ has the form claimed in part (i).

We now prove that 
$F''= \langle w_1|\cdots |w_n|w_nR(c^{-1}(\theta_1)\rangle$ 
is a special $n$-simplex. 
Since $F$ is a special $n$-simplex, Proposition~\ref{char}  gives  
\[w_0^{-1}(v_0) = a_1\mu(c\theta_1) + \dots  + a_n\mu(c\theta_n) \ \mbox{with}\  a_i < 0 \ \mbox{for } 1\le i \le n\]
and, hence,  Proposition~\ref{basischange} gives 
\begin{eqnarray*}
R(c^{-1}\theta_1)w_n^{-1}(v_0)&=& c^{-1}w_1^{-1}(v_0)\\
&=& c^{-1}(a_1\mu(\theta_1)+ a_2\mu(c\theta_2) + \dots  + a_n\mu(c\theta_n))\\
&=& a_1\mu(c^{-1}\theta_1)+ a_2\mu(\theta_2) + \dots  + a_n\mu(\theta_n).
\end{eqnarray*}
As $c = R(\theta_2)\dots R(\theta_n)R(c^{-1}\theta_1)$ and $c^{-1}\theta_1$
is a positive root (by Example~\ref{firstnlastn}),
Proposition~\ref{char} implies that $F''$ is a special $n$-simplex. 

To establish the separation, we recall from Proposition~\ref{supports}~(i) that  $w_0\mu(-c\theta_1)$ is 
orthogonal to the support of $F_0$ and $w_0R(\theta_1)(v_0)=w_1(v_0) \in F_0$.  
Thus, using $\theta\cdot \mu(c\theta) = -1$,
\begin{eqnarray*}
[w_0(v_0)-w_1(v_0)]\cdot w_0\mu(-c\theta_1) &=& [w_0(v_0)-w_0R(\theta_1)(v_0)]\cdot w_0\mu(-c\theta_1) \\
&=& [v_0-R(\theta_1)(v_0)]\cdot \mu(-c\theta_1) \\
&=& 2(\theta_1\cdot v_0)\theta_1\cdot \mu(-c\theta_1) \\
&=& 2(\theta_1\cdot v_0)
 \end{eqnarray*}
which is positive, since $\theta_1$ is a positive root.  
Next,  since $w_n=w_0c$ and $w_n(v_0)-w_1(v_0)$ is orthogonal to $w_0\mu(-c\theta_1)$, we have
\begin{eqnarray*}
[w_{n+1}(v_0)-w_1(v_0)]\cdot w_0\mu(-c\theta_1) 
&=& 
[w_nR(c^{-1}\theta_1)(v_0)-w_1(v_0)]\cdot w_0\mu(-c\theta_1) \\
&=& [w_nR(c^{-1}\theta_1)(v_0)-w_n(v_0)]\cdot w_0\mu(-c\theta_1)\\
&=& [cR(c^{-1}\theta_1)(v_0)-c(v_0)]\cdot \mu(-c\theta_1)\\
&=& c\bigl(-2(c^{-1}\theta_1\cdot v_0)c^{-1}\theta_1\bigr)\cdot \mu(-c\theta_1)\\
&=& -2(c^{-1}\theta_1\cdot v_0)
\end{eqnarray*}
which is negative, since $c^{-1}\theta_1$ is a positive root.  
Thus $w_0(v_0)$ and $w_{n+1}(v_0)$ are on opposite sides of the support of $F_0$.  

(ii)  If  $\theta_1 = \rho_j$ for some $j \in \{1,\dots,n\}$, then $\mu(\theta_1) = \beta_j$  
and
\[w_1(\beta_j) =  w_0R(\theta_1)(\mu(\theta_1))
=w_0(c(\mu(\theta_1))) =
 w_0(\mu(c\theta_1))\]
since $R(\theta_1)(\mu(\theta_1))=c(\mu(\theta_1))$ and $c \circ \mu=\mu\circ c$.
Thus, $-w_1(\beta_j)$ is orthogonal to the support of $F_0$, by Proposition~\ref{supports}.
Since $\beta_j$ is orthogonal to the support of the permutahedron facet whose vertex set 
is $W_j(v_0)$, by Lemma~\ref{negativity}, it follows that 
$-w_1(\beta_j)$ is also orthogonal to the support of the permutahedron facet whose vertex set  is $w_1W_j(v_0)$.
\end{proof} 

\begin{proposition}\label{bottomface}
(i)  If $\theta_n$ is not one of the last $n$ positive roots,
 then $F_n$ is  a facet of exactly one other special $n$-simplex, namely 
\[\sns{w_0R(c(\theta_n))}{w_0}{w_{n-1}},\] 
and the support of $F_n$ separates $w_n(v_0)$ 
from $w_0R(c(\theta_n))(v_0)$.\\
(ii)  If $\theta_n$ is one of the last $n$ positive roots, then 
$F_n$ lies in a facet of $P(W)$.
\end{proposition} 

\begin{proof} 
(i)  This part is analogous to part (i) of Proposition~\ref{topface}.
\vskip .1cm
(ii)  If $\theta_n$ is one of the last $n$ positive roots, then $-c\theta_n$
is one of the first $n$ positive roots and, hence, 
   $\mu(-c\theta_n) = \beta_j$ for some $j\in \{1,\dots,n\}$.  
Since  $w_0(\beta_j) = w_0(\mu(-c\theta_n))$ is orthogonal to  the support of $F_n$, 
by Proposition~\ref{supports}, and to the support of the permutahedron facet 
whose vertex set is $w_0W_j(v_0)$ (as in the proof of part (ii) of Proposition~\ref{topface}), 
it follows that these supports coincide. Thus, $F_n$ lies in the 
permutahedron facet 
whose vertex set is $w_0W_j(v_0)$.
\end{proof}

\begin{proposition}\label{sideface}  If $i \in \{1,\dots,n-1\}$, then
$F_i$  is a facet of exactly one other special $n$-simplex, namely
\[\langle w_0|\cdots |w_{i-1}|w_{i-1}\bar{R}|w_{i+1}|\cdots |w_n\rangle,\]
 where $\bar{R}<_Tw_{i-1}^{-1}w_{i+1}$ is a reflection distinct from $R(\theta_i)$.  
Furthermore, the support of $F_i$ separates $w_i(v_0)$ 
from $w_{i-1}\bar{R}(v_0)$ and, hence, $F_i$ is not contained in a facet of $P(W)$.
\end{proposition}

\begin{proof} Condition (i) of Defiition~\ref{facets} implies that any special $n$-simplex
containing $F_i$ must have the form
\[F'= \langle w_0| \cdots | w_{i-1}|w_{i}'|w_{i+1}|\cdots |w_n\rangle
\mbox{ \  for some } w_i' \in W.\]
The condition that 
\[e<_Tw_0^{-1}w_1 <_T\cdots<_Tw_0^{-1}w_{i-1}<_Tw_0^{-1}w_i'<_Tw_0^{-1}w_{i+1}<_T
\cdots <_Tw_0^{-1}w_n\] 
should be a maximal chain in the non-crossing partition lattice
then implies that $w_i'=w_{i-1}\bar{R}$ for some reflection $\bar{R}\in W$. Thus $F'$
has the claimed form. We now show that  exactly two 
special $n$-simplices of this form exist and contain $F_i$. 

Since $F$ is a special $n$-simplex,  there exist negative scalars $a_1,\dots,a_n$
such that 
$w_n^{-1}(v_0) = a_1\mu(\theta_1)+ \dots +  a_n\mu(\theta_n)$,
by Proposition~\ref{char}.
Let 
\begin{eqnarray*}
v_1 &=& a_1\mu(\theta_1)+ \dots +  a_{i-1}\mu(\theta_{i-1}) \qquad (v_1=0 \mbox{ if } i=1),\\
v_2 &=& a_i\mu(\theta_i)+ a_{i+1}\mu(\theta_{i+1}),\\
v_3 &=& a_{i+2}\mu(\theta_{i+2})+ \dots +  a_n\mu(\theta_n) \qquad (v_3=0 \mbox{ if } i=n-1).
\end{eqnarray*}
Denote  the non-crossing partition $R(\theta_i)R(\theta_{i+1})=w_{i-1}^{-1}w_{i+1}$
by $\sigma$ 
and denote by $W_\sigma$ the rank two, 
reflection subgroup which is generated by those reflections in $W$ which precede $\sigma$.
 Assume that  $\eta_1, \eta_2, \dots , \eta_p$ are the  distinct positive roots 
 corresponding to the reflections in $W_\sigma$ and that they are listed in the
 order inherited from the total order on the roots of $W$.    By Theorem 5.4 of \cite{Lattices}, 
$\{\eta_1, \eta_p\}$ is a simple system for $W_\sigma$ and
\[\sigma = R(\eta_1)R(\eta_p) = R(\eta_p)R(\eta_{p-1}) = \dots  
= R(\eta_3)R(\eta_2) = R(\eta_2)R(\eta_1)\]
is a complete list of the minimal factorisations of $\sigma$. For each factorisation,
$\sigma= R(\eta_j)R(\eta_{\bar{\jmath}} )$,  we have  
\( c= R(\theta_1)\cdots R(\theta_{i-1})R(\eta_j)R(\eta_{\bar{\jmath}} )R(\theta_{i+1})\cdots R(\theta_n).\)
It follows that, in the expression $w_n^{-1}(v_0)=v_1+v_2+v_3$, the vector
$v_2$ can be expressed as a linear combination of $\mu(\eta_j)$ and $\mu(\eta_{\bar{\jmath}})$
and, by equation~(\ref{E-coeffs2}), the coefficients of 
$\mu(\eta_j)$ and $\mu(\eta_{\bar{\jmath}})$ are both nonzero.
Hence, $\mu^{-1}(v_2)$ cannot be parallel to $\eta_j$
for any $j \in \{1,\dots,p\}$. Moreover, Proposition~\ref{char}
implies that 
\[\langle w_0| w_1(v_0)| \cdots | w_{i-1}| w_{i-1}R(\eta_j)|w_{i+1}|\cdots |w_n\rangle\] 
is a special $n$-simplex  if and only if the coefficients of 
$\mu(\eta_j)$ and $\mu(\eta_{\bar{\jmath}})$, in this expression for $v_2$,
are both negative. We show that this occurs for precisely two factorisations of $\sigma$.

Since $a_i,a_{i+1}<0$ and $\{\theta_i, \theta_{i+1}\}\subseteq \{\eta_1, \dots , \eta_{p}\} $,
the vector $\mu^{-1}(v_2)=a_i\theta_i+a_{i+1}\theta_{i+1}$
can be expressed as  a linear combination of $\eta_1$ and $\eta_{p}$ with non-positive coefficients.  In fact, these coefficients must be negative since $\mu^{-1}(v_2)$  is not parallel to any $\eta_j$.   Hence,
 the factorisation $\sigma= R(\eta_1)R(\eta_{p})$
 gives rise to one special $n$-simplex.
Next, since $\eta_2,\dots,\eta_p$ are obtained in this order by
successive rotations of $\eta_1$ through $\pi/p$ in the plane
spanned by $\theta_i$ and $\theta_{i+1}$, it follows 
that there exist unique scalars $a,b<0$
and a unique integer $m \in \{2,\dots,p\}$ 
such that $\mu^{-1}(v_2) =a\eta_m+b\eta_{m-1}$. 
Hence,
 the factorisation $\sigma= R(\eta_m)R(\eta_{m-1})$
 gives rise to another special $n$-simplex.

To establish the separation property, it will be convenient to assume that
$F$   arises from
the factorisation $\sigma=R(\eta_1)R(\eta_p)$ so that
$\theta_i = \eta_1$ and $\theta_{i+1} = \eta_p$.
Letting \(N_i = (\theta_i\cdot v_0)w_0\mu(-c\theta_{i+1})-(\theta_{i+1}\cdot v_0)w_0\mu(-c\theta_i),\)
it follows from Proposition \ref{supports}(iii) that $N_i$
 is orthogonal to the support of $F_i$ and, by Proposition~\ref{dualbases}~(c),
\begin{equation}
\label{dot-prod}
w_0(\epsilon_j)\cdot N_i = \left\{
\begin{array}{cl}
\theta_i\cdot v_0  &   \mbox{if}\ j = i+1,   \\
-\theta_{i+1}\cdot v_0  &    \mbox{if}\ j = i,     \\
 0 &        \mbox{otherwise.}
\end{array}
\right.
\end{equation}
Since $w_{i-1}(v_0),w_0(v_0)\in F_i$, the vector $w_{i-1}(v_0)-w_0(v_0)$ is orthogonal to $N_i$
and, hence,
\begin{eqnarray*}
[w_{i}(v_0)-w_0(v_0)]\cdot N_i
&=&[w_{i}(v_0)-w_{i-1}(v_0)]\cdot N_i\\
&=& w_{i-1}[-2(\theta_i\cdot v_0)\theta_i]\cdot N_i\\
&=& [-2(\theta_i\cdot v_0)w_{0}(\epsilon_{i})]\cdot  N_i\\
&=& 2(\theta_i\cdot v_0)(\theta_{i+1}\cdot v_0) \mbox{ \ by }(\ref{dot-prod}). 
\end{eqnarray*}
The last expression is positive since $\theta_i$ and $\theta_{i+1}$ are positive roots.  
Next, the $i^{\rm th}$ vertex of  the other special $n$-simplex containing $F_i$
is $w'_i(v_0) = w_{i-1}R(\eta_m)(v_0)$, 
with $m$ as above.      As $R(\eta_1)$ permutes the set $\{\eta_2,\dots,\eta_p\}$,
there exist $b_1 \ge 0$ and $b_2>0$ such that
 $R(\eta_1)\eta_m  = b_1\eta_{1}+b_2\eta_p$.
Since $w_{i-1}(\eta_1) = w_0(\epsilon_i)$, $w_i(\eta_p)= w_0(\epsilon_{i+1})$,
and $w_{i-1}(v_0)-w_{0}(v_0)$ is orthogonal to $N_i$, we have
\begin{eqnarray*}
[w'_i(v_0)) -w_0(v_0)]\cdot N_i &=& [w_{i-1}R(\eta_m)(v_0)-w_0(v_0)]\cdot N_i\\
&= &[w_{i-1}R(\eta_m)(v_0)-w_{i-1}(v_0)]\cdot N_i\\
&=& [-2(\eta_m\cdot v_0)w_{i-1}(\eta_m)]\cdot N_i\\
&=& -2(\eta_m\cdot v_0)[w_{i}R(\eta_1)(\eta_{m})\cdot  N_i]\\
&=& -2(\eta_m\cdot v_0)[w_{i}(b_1\eta_{1}+b_2\eta_p)\cdot  N_i]\\
&=& -2(\eta_m\cdot v_0)[(-b_1w_{i-1}(\eta_{1})+b_2w_i(\eta_p))\cdot  N_i]\\
&=& -2(\eta_m\cdot v_0)[(-b_1w_0(\epsilon_{i})+b_2w_0(\epsilon_{i+1}))\cdot  N_i]\\
&=& -2(\eta_m\cdot v_0)(b_1(\theta_{i+1}\cdot v_0)+b_2(\theta_i\cdot v_0)),
\end{eqnarray*}
by equation~(\ref{dot-prod}). This last expression is negative
since $\eta_m$, $\theta_{i}$ and $\theta_{i+1}$ are positive roots while $b_1 \ge 0$ and $b_2>0$.
Thus $w_{i}(v_0)$ and $w'_{i}(v_0)$ 
lie on opposite sides of the supporting hyperplane of $F_i$. 
\end{proof} 
\begin{example}
\label{example-again}
    In the triangulation of Example~\ref{dihedral}, we consider the special 
    $2$-simplex given by 
    $\langle s_1s_2|s_2s_1s_2|s_1s_2s_1s_2\rangle$ and its three faces (See Figure \ref{fig-again}).  Here \[w_0^{-1} = R(\rho_2)w_1^{-1} = R(\rho_2)R(\rho_1)w_2^{-1}\] so that $\theta_1 = \rho_2$ and 
    $\theta_2 = \rho_1$.  As in Proposition~\ref{topface} (ii), $\theta_1 = \rho_2$ is one of the first two positive roots so that the face $F_0$ is on the boundary.  On the other hand, $\theta_2 = \rho_1$ is not among the last two positive roots.  Thus, as in Proposition~\ref{bottomface} (i), the face $F_2$ is shared with the 
    special $2$-simplex 
    $\langle s_2|s_1s_2|s_2s_1s_2\rangle$ since 
\[s_1s_2R(c\rho_1) = s_1s_2R(\rho_3) = s_1s_2(s_1s_2s_1s_2s_1) = s_2.\]
Finally, as  in  Proposition~\ref{sideface}, the face $F_1$ is shared with the special $2$-simplex 
$\langle s_1s_2|s_1s_2s_1|s_1s_2s_1s_2\rangle$.  Here, 
in the terminology of the proof of Proposition~\ref{sideface}, 
the length $2$ non-crossing partition $\sigma$ is given by $\sigma = c$ and the vector $v_2$  is given by 
$v_2 = w_2^{-1}(v_0)$, since $i = 1 = n-i$.  
The vector $v_2$ lies in the cone on $\{-\mu_1, -\mu_2\}$ which gives the initial special $2$-simplex with labelling $R_2R_1$.  Since the cone on $\{-\mu_1, -\mu_2\}$ lies in 
the cone on $\{-\mu_1, -\mu_4\}$ the face $F_1$ also belongs to the special $2$-simplex with labelling $R_1R_4$.
\end{example}

\begin{figure}[htbp]
\begin{center}
\begin{tikzpicture}%
 \tikzstyle{edge}=[->,thick]
\draw [->,blue](8,1)--(7.5,1.5);
\draw [-,blue](7.5,1.5)--(7,2);
\draw [->,blue](8,1)--(8,0);
\draw [-,blue](8,0)--(8,-1);
\draw [->,blue](8,-1)--(7.5,-1.5);
\draw [-,blue](7.5,-1.5)--(7,-2);
\draw [->,blue](7,-2)--(6,-2);
\draw [-,blue](6,-2)--(5,-2);
\draw [->,blue](5,-2)--(4.5,-1.5);
\draw [-,blue](4.5,-1.5)--(4,-1);
\draw [->,blue](4,1)--(4,0);
\draw [-,blue](4,0)--(4,-1);
\draw [->,blue](5,2)--(4.5,1.5);
\draw [-,blue](4.5,1.5)--(4,1);
\draw [->,blue](7,2)--(6,2);
\draw [-,blue](6,2)--(5,2);
\draw [->,blue](8,1)--(7.5,-0.5);
\draw [-,blue](7.5,-0.5)--(7,-2);

\draw [->,blue](7,-2)--(5.5,-1.5);
\draw [-,blue](5.5,-1.5)--(4,-1);
\draw [->,blue](7,2)--(5.5,1.5);
\draw [-,blue](5.5,1.5)--(4,1);
\draw [->,blue](7,2)--(7,0);
\draw [-,blue](7,0)--(7,-2);
\draw [->,blue](7,-2)--(5.5,-0.5);
\draw [-,blue](5.5,-0.5)--(4,1);

\node at (6,-2) [below = 0.1 cm] {$R_1$};

\node at (4.3,-1.5) [below = 0.1 cm] {$R_4$};

\node at (4,0) [left = 0.1 cm] {$R_1$};
\node at (7,0) [left = 0.1 cm] {$R_3$};
\node at (5.5,-0.5) [above = 0.1 cm] {$R_2$};
\node at (5.5,1.5) [below = 0.05 cm] {$c$};
\node at (5.5,-1.5) [below = 0.05 cm] {$c$};

\draw[pattern=north east lines,pattern color=gray!50] (7,-2) -- (4,1) -- (4,-1) -- (7,-2);
 \draw[fill=black] (8 ,1) circle (0.1);
 \draw[fill=black] (8 ,-1) circle (0.1) ;
 \draw[fill=black] (7,-2 ) circle (0.1) node [below=.1]{$s_1s_2$};
 \draw[fill=black] (5,-2 ) circle (0.1)  node [below=.1]{$s_1s_2s_1$};
 \draw[fill=black] (7,2 ) circle (0.1)  node [above=.1]{$s_2$};
 \draw[fill=black] (5,2) circle (0.1) ;
\draw[fill=black] (4,1 ) circle (0.1)  node [left]{$s_2s_1s_2$};
 \draw[fill=black] (4,-1 ) circle (0.1)  node [left]{$s_1s_2s_1s_2$};
\end{tikzpicture}
\end{center}
\caption{The special $2$-simplex of Example \ref{example-again}.}
\label{fig-again}
\end{figure}
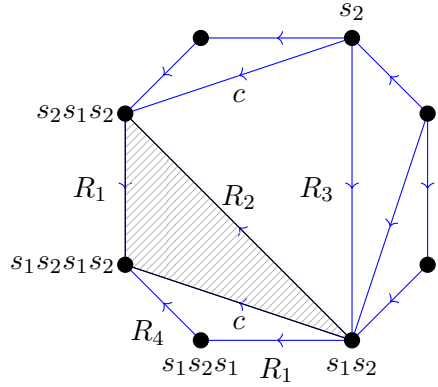

\section{Triangulation criteria}
\label{triangcrit}

In this section we prove that the set $\mathcal{T}$
consisting of all  special $n$-simplices and their faces
is a triangulation of the permutahedron, $P(W)$,
and we prove that this triangulation induces a homotopy equivalence
from $\qsal(W)$ to $\qncp(W)$. 
The proof of the former result  consists of checking that $\mathcal{T}$ satisfies two conditions (from \cite{DRS}) which, in our context, are defined as follows.

\begin{definition}
\label{ICP}
The set $\mathcal{T}$ has the \textbf{intersection cocircuit property}
if and only if (i) whenever $F$ is an  $(n-1)$-simplex in $\mathcal{T}$ 
which is not contained
in a facet of $P(W)$, then there are exactly two $n$-simplices in $\mathcal{T}$
which contain $F$ and (ii) the two additional vertices from these $n$-simplices
lie on opposite sides of the supporting hyperplane of $F$.
\end{definition}

\begin{definition}
\label{GFP}
The set $\mathcal{T}$ has the \textbf{general facet property}
if and only if (i) $P(W)$  has a facet $Q$ which is triangulated by those
simplices of $\mathcal{T}$ which are contained in $Q$ and (ii) no $(n-1)$-simplex 
of this triangulation is a facet
of two distinct $n$-simplices of $\mathcal{T}$.
\end{definition}
\begin{remark}\label{TriConditions} 
If the set $\mathcal{T}$ satisfies the intersection cocircuit property and the general facet property, then it is a triangulation of $P(W)$, by 
Corollaries 4.5.19 and 4.5.20 of \cite{DRS}.
\end{remark}
\begin{theorem} \label{triangulation}
The set $\mathcal{T}$,  consisting of all special $n$-simplices and their faces,  is
a triangulation of $P(W)$. 
\end{theorem}
\begin{proof}   Propositions~\ref{topface}, \ref{bottomface} and \ref{sideface} 
show that $\mathcal{T}$ has the intersection cocircuit property. 
We prove that $\mathcal{T}$ satisfies the general facet property by induction on the rank of $W$, the rank one case being immediate. 
Assume that $n \ge 2$ and that the statement of the theorem is true for groups of rank less than $n$.  
Since the Coxeter diagram of $W$ is a tree, there exists $i \in \{1, \dots , n\}$ such that  
the standard parabolic subgroup, $W_i$, is irreducible.  Let $c' \in W_i$ be the Coxeter element given by 
$c' = s_1\cdots \widehat{s_i} \cdots s_n$ and note that this factorisation is bipartite.   

The reflection group $W_i$ acts irreducibly and effectively on the $(n-1)$-dimensional subspace $U = \beta_i^\perp$.  The simple system $\alpha_1, \dots , \alpha_{i-1},  \alpha_{i+1},  \dots , \alpha_{n}$ 
has dual basis given by $\beta'_1, \dots , \beta'_{i-1},  \beta'_{i+1},  \dots , \beta'_{n}$, where  
$\beta'_j = (\beta_j)_U$, the projection of $\beta_j$ onto $U$ orthogonal to $\beta_i$.
The point 
\[v'_0 = \beta'_1+\cdots +\beta'_{i-1}+\beta'_{i+1}+\cdots+\beta'_{n} = (v_0)_U\]
lies in the fundamental chamber for $W_i$.
By induction, the convex hull of $W_i(v'_0)$ is triangulated by special $(n-1)$-simplices and their faces.  
The translation of this triangulation by 
the vector $[(\beta_i \cdot v_0)/(\beta_i \cdot \beta_i)]\beta_i$ gives a triangulation, $\mathcal{T}_i$, of $Q = P_i$, the convex hull of $W_i(v_0)$.
Applying Proposition~\ref{extendtopboundarysimplex}, we can conclude that 
each $(n-1)$-simplex of $\mathcal{T}_i$ is a facet of exactly one special $n$-simplex in $\mathcal{T}$.   
Thus the triangulation $\mathcal{T}_i$ is the restriction of $\mathcal{T}$  to the facet $Q$ and no special $(n-1)$-simplex of $\mathcal{T}_i$ is a face of two different special $n$-simplices of  $\mathcal{T}$.  Thus $\mathcal{T}$ satisfies the general facet property.  
By Remark~\ref{TriConditions},
$\mathcal{T}$ is a triangulation of $P(W)$.
\end{proof} 
\begin{theorem}   
\label{htyequiv}
The triangulation, $\mathcal{T}$,  defines a map from $\qsal(W)$ to $\qncp(W)$ which is a homotopy equivalence.
\end{theorem} 
\begin{proof}  
The non-crossing partition  $K(\pi, 1)$, $\qncp(W)$, has the structure of a $\Delta$-complex or trisp.  Here 
the $d$-simplex determined by the chain $e<_Tw_1 <_T \cdots <_T w_d$ is assigned to the $(d+1)$-tuple 
$(e,w_1,  \dots , w_d)$ with the ordering on the vertices given by the total reflection length on the $w_i$.
Theorem~\ref{triangulation}  allows us to identify $P(W)$ with a geometric simplicial complex, which is automatically a trisp.  
Here the ordering on the vertices of a simplex $(w_0(v_0), w_1(v_0), \dots , w_d(v_0))$ is given by 
the standard length of the $w_i$.   Thus we can define a map $\Theta$ from $P(W)$ to $\qncp(W)$ 
as the trisp map given by 
\[(w_0(v_0), w_1(v_0), \dots , w_d(v_0)) \mapsto (e, w_0^{-1}w_1, \dots , w_0^{-1}w_d).\]

We show that $\Theta$ induces a well-defined map from $\qsal$ to $\qncp$.    
\vskip .2cm
Let $F$ be a face of $P(W)$ and let $d$ be the dimension of $F$. 
By Proposition~\ref{permfacets}, there is a rank $d$ standard parabolic subgroup 
$W' = \langle s_{i_1}, \dots , s_{i_d}\rangle \subset W$ 
and an element $a \in W$ such that $F$ is the convex hull of $aW'(v_0)$.
Assume that $a$ is the $S$-length minimising element of the coset $aW'$.  
Let $F'$ be a face of $P(W)$ which is identified with $F$ in $\qsal(W)$.  Then 
$F'$ is the convex hull of $bW'(v_0)$ for some $b \in W$ which is the $S$-length minimising representative of its coset $bW'$ and the identification is via the orthogonal transformation $ba^{-1}$.
By Lemma 2.3.4(iv) of \cite{DRS} and Theorem~\ref{triangulation} above, $F$ is triangulated by $d$-simplices which are the intersections of $F$ with special $n$-simplices of the triangulation of $P(W)$.  
Let $\sigma$ be such a $d$-simplex and suppose $\sigma$ is the convex hull of 
\[aw(v_0), awu_1(v_0), \dots , awu_d(v_0),\]
where $w \in W'$, $e\lessdot_T u_1\lessdot_T  \dots \lessdot_T u_d = c'$ and $c' = s_{i_1}\cdots s_{i_d} \in W'$.    By the choice of $a$ and the definition of a special $n$-simplex,
\[l_S(awu_i) = l_S(a)+l_S(w)+l_T(u_i) \ \mbox{for} \ 1\le i \le d.\]
In $\qsal(W)$, $\sigma$ is identified with the convex hull of 
\[bw(v_0), bwu_1(v_0), \dots , bwu_d(v_0),\]
in the face $F'$.  
By Proposition~2.4.4 of \cite{BjBr},
\[l_S(bwu_i) = l_S(b)+l_S(w)+l_T(u_i) \ \mbox{for} \ 1\le i \le d.\]
We now apply Proposition~\ref{extendboundarysimplex2} to conclude that $bw(v_0), bwu_1(v_0), \dots , bwu_d(v_0)$ 
are the vertices of a $d$-face, $\sigma'$, of a special $n$-simplex.  Finally the images $\Theta(\sigma)$ and $\Theta(\sigma')$ in $\qncp(W)$ are both equal to the $(d+1)$-tuple $(e, u_1,  \dots , u_d)$, so that $\Theta$ induces a well defined map from $\qsal(W)$ to $\qncp(W)$.

The map induced by $\Theta$ is a homotopy equivalence by the proof of Theorem~1B.8 of \cite{H}, since, by Lemma ~\ref{iso-fg}, the induced map on fundamental groups is an isomorphism.
\end{proof} 
\bibliographystyle{plain}
\bibliography{triangle}

@article {Bessis,
    AUTHOR = {Bessis, David},
     TITLE = {The dual braid monoid},
   JOURNAL = {Ann. Sci. \'{E}cole Norm. Sup. (4)},
  FJOURNAL = {Annales Scientifiques de l'\'{E}cole Normale Sup\'{e}rieure.
              Quatri\`eme S\'{e}rie},
    VOLUME = {36},
      YEAR = {2003},
    NUMBER = {5},
     PAGES = {647--683},
      ISSN = {0012-9593},
   MRCLASS = {20F36 (57M07)},
  MRNUMBER = {2032983},
MRREVIEWER = {Valeriy\ G.\ Bardakov},
       DOI = {10.1016/j.ansens.2003.01.001},
       URL = {https://doi.org/10.1016/j.ansens.2003.01.001},
}

@article{BessisAnnals,
 author = {Bessis, David},
 title = {Finite complex reflection arrangements are {{\(K(\pi,1)\)}}},
 fjournal = {Annals of Mathematics. Second Series},
 journal = {Ann. Math. (2)},
 issn = {0003-486X},
 volume = {181},
 number = {3},
 pages = {809--904},
 year = {2015},
 language = {English},
 doi = {10.4007/annals.2015.181.3.1},
 zbMATH = {6446405},
 Zbl = {1372.20036}
}

@book {BjBr,
    AUTHOR = {Bj\"orner, Anders and Brenti, Francesco},
     TITLE = {Combinatorics of {C}oxeter groups},
    SERIES = {Graduate Texts in Mathematics},
    VOLUME = {231},
 PUBLISHER = {Springer, New York},
      YEAR = {2005},
     PAGES = {xiv+363},
      ISBN = {978-3540-442387; 3-540-44238-3},
   MRCLASS = {05-01 (05E15 20F55)},
  MRNUMBER = {2133266},
MRREVIEWER = {Jian-yi Shi},
}

@article {BjWa,
    AUTHOR = {Bj\"{o}rner, Anders and Wachs, Michelle L.},
     TITLE = {Generalized quotients in {C}oxeter groups},
   JOURNAL = {Trans. Amer. Math. Soc.},
  FJOURNAL = {Transactions of the American Mathematical Society},
    VOLUME = {308},
      YEAR = {1988},
    NUMBER = {1},
     PAGES = {1--37},
      ISSN = {0002-9947},
   MRCLASS = {05A99 (06F99 20B30 20F99)},
  MRNUMBER = {946427},
MRREVIEWER = {Joseph Kung},
       DOI = {10.2307/2000946},
       URL = {https://doi.org/10.2307/2000946},
}

@inproceedings{BWKayPie,
	Author = {T.~Brady and C.~Watt},
	Booktitle = {Proceedings of the {C}onference on {G}eometric and {C}ombinatorial {G}roup {T}heory, {P}art {I} ({H}aifa, 2000)},
	Journal = {Geom. Dedicata},
	Pages = {225--250},
	Title = {{$K(\pi,1)$}'s for {A}rtin groups of finite type},
	Volume = {94},
	Year = {2002}}

@article{Lattices,
	Author = {Brady, T. and Watt, C.},
	Coden = {TAMTAM},
	Doi = {10.1090/S0002-9947-07-04282-1},
	Fjournal = {Transactions of the American Mathematical Society},
	Issn = {0002-9947},
	Journal = {Trans. Amer. Math. Soc.},
	Mrclass = {20F55 (05E15)},
	Mrnumber = {2366971 (2008k:20088)},
	Mrreviewer = {Alessandro Conflitti},
	Number = {4},
	Pages = {1983--2005},
	Title = {Non-crossing partition lattices in finite real reflection groups},
	Url = {http://dx.doi.org/10.1090/S0002-9947-07-04282-1},
	Volume = {360},
	Year = {2008}}

@book {Davis,
    AUTHOR = {Davis, Michael W.},
     TITLE = {The geometry and topology of {C}oxeter groups},
    SERIES = {London Mathematical Society Monographs Series},
    VOLUME = {32},
 PUBLISHER = {Princeton University Press, Princeton, NJ},
      YEAR = {2008},
     PAGES = {xvi+584},
      ISBN = {978-0-691-13138-2; 0-691-13138-4},
   MRCLASS = {20F55 (05B45 05C25 51-02 57M07)},
  MRNUMBER = {2360474},
MRREVIEWER = {Ralf Koehl},
}

@article{DelucchiPaoliniSalvetti,
 author = {Delucchi, Emanuele and Paolini, Giovanni and Salvetti, Mario},
 title = {Dual structures on {Coxeter} and {Artin} groups of rank three},
 fjournal = {Geometry \& Topology},
 journal = {Geom. Topol.},
 issn = {1465-3060},
 volume = {28},
 number = {9},
 pages = {4295--4336},
 year = {2024},
}

@article{DSBW,
  title={Permutahedron Triangulations via Total Linear Stability and the Dual Braid Group}, 
      author={Colin Defant and Melissa Sherman-Bennett and Nathan Williams},
      year={2025},
      journal={arXiv preprint arXiv:2509.11497},

}

@book {DRS,
    AUTHOR = {De Loera, Jes\'us A. and Rambau, J\"org and Santos, Francisco},
     TITLE = {Triangulations},
    SERIES = {Algorithms and Computation in Mathematics},
    VOLUME = {25},
      NOTE = {Structures for algorithms and applications},
 PUBLISHER = {Springer-Verlag, Berlin},
      YEAR = {2010},
     PAGES = {xiv+535},
      ISBN = {978-3-642-12970-4},
   MRCLASS = {52B55 (05C10 52B05 57Q15 68U05)},
  MRNUMBER = {2743368},
       DOI = {10.1007/978-3-642-12971-1},
       URL = {https://doi.org/10.1007/978-3-642-12971-1},
}

@book{H,
	Author = {A.~Hatcher},
	Isbn = {0-521-79160-X; 0-521-79540-0},
	Mrclass = {55-01 (55-00)},
	Mrnumber = {1867354 (2002k:55001)},
	Mrreviewer = {Donald W. Kahn},
	Pages = {xii+544},
	Publisher = {Cambridge University Press, Cambridge},
	Title = {Algebraic topology},
	Year = {2002}}

@book{Hum90,
	Author = {Humphreys, J. E.},
	Doi = {10.1017/CBO9780511623646},
	Isbn = {0-521-37510-X},
	Mrclass = {20-02 (20F32 20F55 20G15 20H15)},
	Mrnumber = {1066460 (92h:20002)},
	Mrreviewer = {Louis Solomon},
	Pages = {xii+204},
	Publisher = {Cambridge University Press, Cambridge},
	Series = {Cambridge Studies in Advanced Mathematics},
	Title = {Reflection groups and {C}oxeter groups},
	Url = {http://dx.doi.org/10.1017/CBO9780511623646},
	Volume = {29},
	Year = {1990}}

@book{K,
	Address = {Berlin},
	Author = {Kozlov, D.},
	Doi = {10.1007/978-3-540-71962-5},
	Isbn = {978-3-540-71961-8},
	Mrclass = {55-02 (05C15 05C25 06A07 52B70 55U10 57-02)},
	Mrnumber = {2361455 (2008j:55001)},
	Mrreviewer = {Rade {\v{Z}}ivaljevi{{\'c}}},
	Pages = {xx+389},
	Publisher = {Springer},
	Series = {Algorithms and Computation in Mathematics},
	Title = {Combinatorial algebraic topology},
	Url = {http://dx.doi.org/10.1007/978-3-540-71962-5},
	Volume = {21},
	Year = {2008}}

@article{PaoliniSalvetti,
 author = {Paolini, Giovanni and Salvetti, Mario},
 title = {Proof of the {{\(K(\pi,1)\)}} conjecture for affine {Artin} groups},
 fjournal = {Inventiones Mathematicae},
 journal = {Invent. Math.},
 issn = {0020-9910},
 volume = {224},
 number = {2},
 pages = {487--572},
 year = {2021}
}

@article{Salvetti87,
 author = {Salvetti, M.},
 title = {Topology of the complement of real hyperplanes in {{\(\mathbb C^N\)}}},
 fjournal = {Inventiones Mathematicae},
 journal = {Invent. Math.},
 issn = {0020-9910},
 volume = {88},
 pages = {603--618},
 year = {1987},
 language = {English},
 doi = {10.1007/BF01391833},
 url = {https://eudml.org/doc/143468},
 zbMATH = {3956003},
 Zbl = {0594.57009}
}

@article{Salv94,
	Author = {Salvetti, M.},
	Fjournal = {Mathematical Research Letters},
	Issn = {1073-2780},
	Journal = {Math. Res. Lett.},
	Mrclass = {52B30 (20F36 20F55)},
	Mrnumber = {1295551 (95j:52026)},
	Mrreviewer = {Luis Paris},
	Number = {5},
	Pages = {565--577},
	Title = {The homotopy type of {A}rtin groups},
	Volume = {1},
	Year = {1994}}

@article {Steinberg,
    AUTHOR = {Steinberg, Robert},
     TITLE = {Finite reflection groups},
   JOURNAL = {Trans. Amer. Math. Soc.},
  FJOURNAL = {Transactions of the American Mathematical Society},
    VOLUME = {91},
      YEAR = {1959},
     PAGES = {493--504},
      ISSN = {0002-9947},
   MRCLASS = {50.00 (10.00)},
  MRNUMBER = {0106428},
MRREVIEWER = {H. S. M. Coxeter},
       DOI = {10.2307/1993261},
       URL = {https://doi.org/10.2307/1993261},
}

@book {Ziegler,
    AUTHOR = {Ziegler, G\"unter M.},
     TITLE = {Lectures on polytopes},
    SERIES = {Graduate Texts in Mathematics},
    VOLUME = {152},
 PUBLISHER = {Springer-Verlag, New York},
      YEAR = {1995},
     PAGES = {x+370},
      ISBN = {0-387-94365-X},
   MRCLASS = {52Bxx},
  MRNUMBER = {1311028},
MRREVIEWER = {Margaret M. Bayer},
       DOI = {10.1007/978-1-4613-8431-1},
       URL = {https://doi.org/10.1007/978-1-4613-8431-1},
}
\end{document}